\pdfoutput=1

\documentclass[11pt,a4paper,notitlepage,oneside]{amsart}
\usepackage{etoolbox}
\usepackage{microtype}
\usepackage{amsmath}
\usepackage{amsthm}
\usepackage[psamsfonts]{amssymb}
\usepackage[T1]{fontenc}
\usepackage{lmodern}
\usepackage[english]{babel}
\usepackage[utf8]{inputenc}
\usepackage{textcomp}
\usepackage{newunicodechar}
\newunicodechar{τ}{\ensuremath{\tau}}
\newunicodechar{ℝ}{\ensuremath{\R}}
\usepackage{enumitem}
\usepackage{mathtools}
\usepackage{caption,subfig}
\usepackage[normalem]{ulem}
\usepackage{xcolor,soul}
\usepackage{scalerel}
\usepackage{relsize}
\usepackage[normalem]{ulem}
\usepackage{booktabs}

\usepackage[autostyle,italian=guillemets]{csquotes}
\usepackage[style=alphabetic,giveninits=true,maxbibnames=99,backend=biber]{biblatex}
\allowdisplaybreaks 

\usepackage{pgf,tikz,pgfplots}

\usepackage{mathrsfs}
\usetikzlibrary{arrows}

\usepackage{fancyhdr}
\usepackage{tikz-cd}
\usepackage{amscd}
\usepackage[a4paper, bottom=3cm, left=3cm,right=3cm]{geometry}
\usepackage{hyperref}
\hypersetup{hidelinks}

\usepackage{subfiles}
\usepackage{PhDefs}

\usepackage{mathtools}
\usepackage{dsfont}
\numberwithin{equation}{section} 
\theoremstyle{plain}

\newtheorem{thm}{Theorem}
\numberwithin{thm}{section}

\newtheorem{lemma}[thm]{Lemma}
\newtheorem{prop}[thm]{Proposition}
\newtheorem{cor}[thm]{Corollary}

\theoremstyle{definition}
\newtheorem{de}[thm]{Definition}

\newtheorem{rem}[thm]{Remark}

\theoremstyle{remark}
\newcounter{exc}

\numberwithin{exc}{subsection}

\makeatletter
\renewenvironment{proof}[1][\textit{Proof}]{\par  
  \pushQED{\qed}%
  \normalfont \topsep6\p@\@plus6\p@\relax
  \trivlist
  \item[\hskip\labelsep
        \upshape
    #1\@addpunct{:}]\ignorespaces
}{%
  \popQED\endtrivlist\@endpefalse
}
\makeatother

\author{Ingmar Metzler}
\address{
	Department of Mathematics, ETH Zurich, 8092 Zurich, Switzerland
}
\email{ingmar.metzler@math.ethz.ch}

\title{Symmetric square type L--series}

\begin{document}
	\begin{abstract}
		We construct symmetric square type \(L\)-series 
		for vector-valued modular forms transforming under 
		the Weil representation associated to a discriminant form. 
		We study Hecke operators and integral representations to investigate their properties, 
		deriving functional equations and infinite product expansions. 
	\end{abstract}
	\maketitle
	\nocite{Metzler2024}
	\tableofcontents
	\section{Introduction}

Symmetric square $L$-series have already been constructed by Shimura \cite{ShimuraMF_HalfIntegralWeight}. These
were instrumental in establishing the \emph{Shimura lift} -- a pivotal correspondence
between modular forms of half-integral and integral weights. This groundbreaking
work catalysed further investigations, including Shimura's own extensions \cite{Shimura1975}
and generalisations to higher rank groups in \cite{Bump1992}. The present work addresses the
natural extension of these $L$-series to the setting of vector-valued modular
forms associated to the Weil representation. More explicitly, let $(L,\qq)$
denote a non-degenerate even lattice, $L'$ denote its dual, \(\efr_{\mu}\) a basis
of the group algebra \(\C[L'/L]\), and $k \in \tfrac{1}{2} \Z$ a half-integral
number. Further, if $\MF_{L,k}$ represents the space of $\C[L'/L]$ valued
holomorphic modular forms of weight $k$ for the associated Weil representation,
then any form~$f \in \MF_{L,k}$ has a Fourier expansion
\begin{equation}\label{eq:Intro_VVMF_Fourierexp}
  f(\tau) = \sum_{\mu \in L'/L} \sum_{m \in \Z + \qq(\mu)} a(\mu,m) e^{2 \pi i \tau m} \efr_{\mu}.
\end{equation}
In extending symmetric square $L$-series from the scalar to the vector-valued
setting, one might initially consider associating an $L$-series with each
component function of~$f$ individually, meaning for fixed $\mu \in L'/L$ in \eqref{eq:Intro_VVMF_Fourierexp}.
However, closer examination reveals that the natural analogues exhibit
intertwining of the components. For clarity of exposition, we present a special
instance of the series discussed in the main body of the paper. For an
anisotropic vector $\posvec \in L'$ and \(\overline{\posvec} \in L'/L\) its
projection, we define
\begin{equation}\label{eq:Intro_L_series}
  L_{(\overline{\posvec},\qq(\posvec))}(f,s) = \sum_{n \in \N} \frac{a(n \overline{\posvec}, n^2 \qq(\posvec))}{n^s}
\end{equation}
for $s \in \C$. These series will serve as cornerstone elements in the forthcoming
work on automorphic lifts where they appear as cycle integrals of theta
liftings. The goal of the present work is to establish several analytic
properties of symmetric square $L$-series associated to vector-valued modular
forms.

To start with existing literature, we establish bounds of Fourier coefficients
(cf.~Cor.~\ref{cor:VVMF_bound_FC_for_Lseries}), which imply absolute convergence of~$L_{\posvec}(f,s)$ with
respect to the parameter~$s$ in case $\Re(s) > k + 1/2$, thereby ensuring
holomorphicity. We achieve sharper bounds in Corollary~\ref{cor:LellNfsabsoluteconvergence}. In Subsection~\ref{ssec:Lseries_RankinSelberg} we
realise generalisations of the $L$-series appearing in~\eqref{eq:Intro_L_series} as a Rankin--Selberg
convolution and establish the convergence and analyticity of its unfolded
counterpart within specified domains (cf.\ Proposition~\ref{prop:RankinSelbergVVMFThetaposdefsplitcrossEisenstein}). More precisely, we
find the following.
\begin{thm}
  Assume $k \geq 2$, $L = \Z\posvec \oplus L_2$ for some lattice $L_{2}$ and
  $\posvec \in L$ with $\qq(\posvec) > 0$. Let $E_{L_2,k_2}(\argdot,s)$ be the
  Eisenstein series from Definition~\ref{de:VVEisensteinseriesnonholomorphic}, assume $k_2 = k - 1/2$ and write
  $\posvecprim = \tfrac{\posvec}{2\qq(\posvec)}$. Then we have the identity
  \begin{equation}
    \int_{\Fd_{\Gamma(1)}} \left\langle f , \Theta_{\Z \posvec} \otimes E_{L_2,k_2}(\argdot,s) \right\rangle \cdot \Im^k \di \mu	
    =	\frac{\Gamma(\overline{s}+k-1)}{(4 \pi)^{\overline{s}+k-1}} \cdot L_{(\overline{\posvecprim},\qq(\posvecprim))}(f,\overline{s} + k - 1)
  \end{equation}
  of holomorphic functions in $\overline{s}$ for $\Re(s) > \tfrac{5 - 2k }{4}$.
  Consequently, \(L_{(\overline{\posvecprim},\qq(\posvecprim))}(f,s)\) admits
  meromorphic continuation and a functional equation with meromorphic
  coefficients \(c(\mu,s)\):
  \begin{equation}\label{eq:fun_equ_intro}
    L_{(\overline{\posvecprim},\qq(\posvecprim))}(f,s) = \frac{1}{2} \sum_{\mu \in \Iso(L_{2}'/L_{2})} c_{\eta}(\mu,s) L_{(\overline{\posvecprim} + \mu , \qq(\posvecprim))}(f,1-k-s). 
  \end{equation}
\end{thm}

A key feature of Shimura's $L$-series is their admission of infinite product
expansions when they are attached to Hecke eigenforms. We replicate this
property in the vector-valued setting through careful analysis of Hecke
operators, building upon the work of Bruinier and Stein (cf.~\cite{BruinierStein, SteinHecke}). In the
process, a bound on eigenvalues is derived (cf.~Lemma~\ref{lem:VVMF_Kohnenbound}). We state a
simplified version of Theorem~\ref{cor:Lell_complete_product}.

\begin{thm}\label{thm:Intro_Lseries_product}
  Assume $L$ to be maximal with $2 \nmid \level(L)$ and select a simultaneous
  eigenform $f \in \CF_{L,k}$ of all $\Hophk(p^2)$ with eigenvalue $\Hev_{p}$ that
  is invariant under $\OG(L'/L)$. For an element $\posvec \in L'$ such that
  $\level(L) \cdot \qq(\posvec)$ is square free, if we denote
  $\lambda \coloneq \overline{\posvec} \in L'/L$, then the following product expansion
  holds:
  \begin{align*}
    & \sum_{n \in \N} \frac{a(n\lambda,n^2\qq(\posvec))}{n^s} \\
    =\;	&a(\lambda,\qq(\posvec)) \cdot \prod_{p \nmid \level(L)}
          \frac{1 + \delta_{p \nmid \qq(\posvec)} \frac{G_{\Df}(1)}{G_{\Df}(p)} \cdot p^{k-1-s} }{ 1 - \left( \frac{\Hev_{p}}{p^{k-1}} - (1-p^{-1}) \frac{G_{\Df}(1)}{G_{\Df}(p)} \right)  \cdot  p^{k-1-s} + p^{2(k-1-s)} } 
    \\
    &\cdot \prod_{p \mid \level(L)}
      \frac{1 + K_{L,p} 
      \begin{cases}
        \delta_{\lambda_{p} \neq 0} (1 - p^{-1}) + \delta_{\lambda_{p} = 0} \cdot \delta_{p \nmid \qq(\posvec)}  , & 2 \mid {R_{p}} \\ 
        - \delta_{\lambda_{p} = 0} \cdot p^{-1/2} \left(\frac{-\qq(\posvec)}{p}\right) \epsilon_{p}, & 2 \nmid {R_{p}}
      \end{cases} 
      \cdot p^{k-1-s} + C(\lambda_{p}) \cdot p^{2(k-1-s)}}
      {1 - \left( \frac{\Hev_{p}}{p^{k-1}} - \delta_{2 \mid {R_{p}}} (1-p^{-1}) K_{L,p} \right) p^{k-1-s} + p^{2(k-1-s)}}.
  \end{align*}
  The infinite product over primes $p \nmid \level(L)$ converges absolutely and does
  not vanish for $\Re(s) > k + 1$. Here, $G_{\Df}$ and $K_{L,p}$ are certain Gauss
  sums, $\lambda_p$ denotes the projection of $\lambda \in L'/L$ to its $p$-component,
  $C(\lambda_{p})$ is a specified integer vanishing if $\lambda_{p} = 0$, $\chi_{\Df}$ is a
  quartic character and $\epsilon_{p}$ equals $1$ or $i$. Further, the quantity $R_{p}$
  is presented in Definition~\ref{de:nk_and_R}.\end{thm}

Furthermore, we establish the non-vanishing of the remaining rational factors in
$p^{-s}$ for primes $p$ dividing $\level(L)$ in a range dependent on the
$p$-rank of $L'/L$ (cf.\ Proposition~\ref{prop:VVMF_Lseries_eigenform_extract_factor_alloddprimes}).
		
\subsection*{Acknowledgments}
I am grateful to my supervisor Jan Bruinier for his guidance and numerous
fruitful discussions exceeding the scope of this project. Also, I would like to
thank Markus Schwagenscheidt and Paul Kiefer for substantial improvements to the
manuscript. Further, I would like to extend my gratitude to Özlem Imamoglu, Jens
Funke, and Oliver Stein for helpful remarks. The author was supported by the CRC
TRR~$326$ `GAUS', project number~$444845124$, the LOEWE project `USAG', and ETH
Zurich.

	\section{Setting}

We briefly present the setting for this paper, including lattices, 
discriminant forms, the Weil representation, and elliptic modular forms. 

We let $(L,\qq)$ denote a non-degenerate even quadratic lattice of signature $(\possig,\negsig)$ and denote its bilinear form by $\bilf$ such that $\qq(x) \coloneq \bilf(x,x)/2$. 
The dual lattice $L' \subset L \otimes_{\Z} \Q$ represents $\Hom_{\Z}(L,\Z)$ via $\bilf$ and we write $\level(L)$ 
for the level of the lattice, being the smallest positive integer such $(L',\level(L) \cdot \qq)$ 
is even. Recall the inclusion $L \leq L'$ and note that $\Df \coloneq L'/L$ is a finite abelian group 
whose order is the absolute value of the determinant of a representing matrix of $\bilf$. 
Further, the quadratic form $\qq$ descends to a quadratic form $\qd : \Df \to \Q/\Z$ that is non-degenerate 
and such a pair $(\Df,\qd)$ is called a \emph{discriminant form}. 
In fact, any discriminant form arises from a non-degenerate even lattice in the above fashion. 
If we set $\sig(\Df) = \possig - \negsig \mod 8$ for some lattice $L$ inducing $\Df$, 
this quantity is well defined by Milgram's formula. 
For a natural number $n \in \N$ denote by $L(n)$ the pair $(L, n \cdot \qq)$ and set~$\Df(n) \coloneq L(n)'/L(n)$. 

Additionally, write $\Dftors{n}$ for the $n$-torsion in $\Df$ and $\Dfmult{n}$ for the $n$-multiples. We necessarily have the exact sequence 
\begin{equation}\label{eq:Df_torsion_mults}	
	\begin{tikzcd}
		0 \arrow[]{r}{} 		& \Dftors{n} \arrow[]{r}{\iota}		&	 \Df \arrow[]{r}{\cdot n}		& 	\Dfmult{n}  \arrow[]{r}{}		&	0 
	\end{tikzcd}
\end{equation}	
and find that $\Dfmult{n}$ is exactly the orthogonal complement of $\Dftors{n}$. 
We denote by $\C[\Df]$ the group algebra of $\Df$ with standard basis $\efr_{\lambda}$ for $\lambda \in \Df$.
There is a representation of $\SL_{2}(\Z)$ on the group algebra $\C[\Df]$ first described by Schoeneberg. 
In fact, to obtain a well-defined representation when $\sig(\Df)$ is odd, 
it must be considered on the metaplectic extension, rather than on $\SL_{2}$ directly. 
Let $\MGLtwop(\R)$ denote the metaplectic double cover of $\GL_{2}(\R)^{+}$. 
A model is given by pairs $(\gamma,\phi)$, where $\gamma \in \GL_{2}(\R)^{+}$ and $\phi : \Ha \to \C$ 
is a holomorphic square root of the factor of automorphy $j(\gamma,\tau)$, meaning $\phi^2(\tau) = j(\gamma,\tau)$. 
Multiplication is then declared by 
\[
	(\gamma, \phi(\tau)) (\gamma',\phi'(\tau)) = (\gamma \gamma', \phi(\gamma' \tau) \phi'(\tau)). 
\]
Projection to the first component \( \MGLtwop \ni (\gamma, \phi) \mapsto \gamma \in \GL_{2}^{+}(\R)\) defines a covering. 
The action of an element $\abcds = \gamma \in \GL_{2}(\R)$ on the upper half plane $\Ha$ is given by Möbius transform 
\(\tau \mapsto (a \tau + b)/(c \tau + d)\) with diagonal invariance group $\R^{\times} \into \GL_{2}^{+}(\R)$ and induces an action on $\MGLtwop(\R)$ via pullback. 
For the preimage of $\Gamma(1) \coloneq \SL_{2}(\Z)$ in $\MGLtwop(\R)$, we write $\Mp_{2}(\Z)$. 
This group is generated by the two elements $\Telmp = \left( \sm{1}{1}{0}{1}, 1\right)$ 
and $\Selmp = \left( \sm{0}{-1}{1}{0}, \sqrt{\tau} \right)$, where $\sqrt{\tau}$ denotes the standard holomorphic square root. 
These elements project to the standard generators of $\SL_{2}(\Z)$ and $\Selmp^2 = \Zelmp = \left(- \sm{1}{0}{0}{1} , i \right)$ 
generates the centre of $\Mp_{2}(\Z)$. 
\\
We now turn towards describing the Weil representation. It suffices to describe the action of the Weil representation $\rho_L$ associated to $L$ on the generators $\overline{T}$ and $\overline{S}$ of $\Mp_{2}(\Z)$. 
	We abbreviate $e(\argdot) \coloneq \exp(2 \pi i \argdot)$ and note that these operations are given by 
\begin{align}
	\dwrep_L(\Telmp) \mathfrak{e}_\lambda 	&=	e(\qd(\lambda)) \cdot \mathfrak{e}_\lambda,	\label{eq:op_T_via_rho}\\
	\dwrep_L(\Selmp) \mathfrak{e}_\lambda 	&=	\frac{e(-\sig(\Df)/8)}{\sqrt{|\Df|}} \cdot \sum_{\mu \in \Df} e(-\beta(\lambda,\mu)) \mathfrak{e}_\mu \label{eq:op_S_via_rho}.
\end{align}
It is apparent that the representation $\rho_{L}$ only depends on the discriminant form $\Df$.
\begin{de}
	Let $(\Df,\qd)$ be a discriminant form. 
	Then $\dwrep_{\Df} \coloneq \dwrep_L$ is a unitary representation of $\Mp_{2}(\Z)$ on $\C[\Df]$ that is fully determined by 
	\eqref{eq:op_T_via_rho} and \eqref{eq:op_S_via_rho}. 
\end{de}
Note that the generator of the centre $\overline{Z}$ operates for $\lambda \in \Df$ as 
		\begin{align}
			\dwrep_\Df(\overline{Z}) \mathfrak{e}_\lambda = e(-\sig(\Df)/4) \cdot \mathfrak{e}_{-\lambda}. \label{eq:op_Z_via_rho}
		\end{align}
		As a consequence, the Weil representation $\rho_\Df$ factors through $\Mp_{2}(\Z) / \langle\overline{Z}^2\rangle \simeq \SL_2(\Z)$, 
		if, and only if, $\sig(\Df) \equiv 0 \mod 2$.
		Further, writing $N \coloneq \level(\Df)$ we find that $\dwrep_{\Df}$ is trivial on the congruence subgroup $\Gamma(N)$ for even rank, 
	so that it factors through the finite group
	\[
		\SL_{2}(\Z /N\Z) \simeq \Gamma(1) / \Gamma(N).
	\]
	In case of odd rank, the oddity formula \cite[Chap.~15 p.~383 (30)]{ConwaySloane1998} implies $4 \mid N$, in particular, $\Df$ contains $2$-adic Jordan components. 
	In this case, there is a well known section 	\begin{align}\label{eq:section_Gamma4_metaplectic}
		s: \Gamma(4) \to \Mp_2(\Z), \qquad \gamma = \abcd \mapsto \left( \abcd , \epsilon_{d}^{-1} \left( \frac{c}{d} \right) \sqrt{j(\gamma, \tau)} \right)
	\end{align}
	with $\epsilon_{d} = 1$ or $i$, depending on whether $d \equiv 1$ or $3 \mod 4$. 
		The argument at the end of \cite[Thm.~5.4 p.~330]{Borcherds2000} yields that $\dwrep_{\Df}$ is trivial on $s(\Gamma(N))$ and factors through the central extension of $\SL_{2}(\Z/N\Z)$ by $\{\pm 1\}$ given by 
	\[
		\Mp_{2}(\Z) / s(\Gamma(N)). 
	\]
In fact, the section in \eqref{eq:section_Gamma4_metaplectic} may be considered on $\Gamma_0(4)$, 
provided a larger central extension $\mathcal{G}$ of $\SL_{2}(\R)$ is considered which is given in \cite[p.~444]{ShimuraMF_HalfIntegralWeight}.

The natural projection $\Mp_{2}(\Z) \to \SL_{2}(\Z)$ has the section $\gamma \mapsto \tilde{\gamma} \coloneq (\gamma, \sqrt{j(\gamma,\tau)})$, 
where we selected the standard branch of the holomorphic square root. Do note that this mapping is not a group homomorphism. 
Further, we require the following piece of notation for the computation of Fourier expansions of non-holomorphic Eisenstein series. 

\begin{de}\label{de:WeilRepCoeffs}
	For $\lambda,\mu \in \Df$ and $(\gamma,\phi) \in \Mp_{2}(\Z)$, we introduce the following notation for the $\lambda,\mu$-\emph{coefficient}\index{Weil!representation!coefficient} of the Weil representation $\dwrep$:
	\begin{equation}
		\dwrep_{\lambda,\mu}(\gamma,\phi) \coloneq \langle \rho_{L}(\gamma,\phi) \efr_{\lambda} , \efr_{\mu} \rangle.
	\end{equation}
\end{de}

Before continuing we briefly recall the notion of a classical modular form. 
\begin{de}\label{de:Peterssonslash}
	Let $\gamma \in \GL_{2}^{+}(\R)$ and $f : \Ha \to \C$ be a function. Define 
	\[
		f\vert_{k} \gamma (\tau) \coloneq \det(\gamma)^{k/2} j(\gamma, \tau)^{-k}(\tau) f(\gamma \tau).
	\]
\end{de}

Finally, a holomorphic modular form $f : \Ha \to \C$ of weight $k$ for some congruence subgroup $\Gamma \leq \Gamma(1)$ with character $\chi : \Gamma \to \mathbb{T}$ 
is a holomorphic function such that for all $\gamma \in \Gamma$ we have $f \vert_{k} \gamma = \chi(\gamma) f$ and $f$ is holomorphic at the cusps of $\Gamma$. 
The space of such functions is denoted by $\MF_{k}(\Gamma,\chi)$ and the subspace of elements vanishing at the cusps is denoted by $\CF_{k}(\Gamma,\chi)$ 
and called the space of cusp forms. 
For the technical modifications required for forms of half-integral weight, we refer to \cite[p. 444]{ShimuraMF_HalfIntegralWeight}. 
Note that Definition~\ref{de:Peterssonslash} may also be generalised to the extension $\mathcal{G}$ of $\GL_{2}^{+}(\R)$ that Shimura uses.

\subsection{Gauss sums}

This is a complementary section containing different Gauss sums appearing in
section~\ref{sec:Hecketheory}. The presented sums and some relations are already found in~\cite[Sec. 4]{SteinHecke}
including further computations which we will not repeat.

\begin{de}\label{de:Gauss_sums}
  Let $(L,\qq)$ be a quadratic $\Z$-lattice of rank $\dimV \in \N$. Let $n \in \N$,
  $p$ be a prime, and consider $h \in \Z$ plus a Dirichlet Character $\chi$ modulo
  $n$. Define
  \begin{align}
    G_{L,p}(n,h) 							&:= \sum_{v \in L/p^nL} e\left( \frac{h}{p^n} \qq(v) \right),  \\
    g_p\left[ n , \chi , h\right] 	&:= \sum_{k \in \Z/p^n \Z} \chi(k) e \left( \frac{hk}{p^n} \right). 
  \end{align}
\end{de}

For odd $d \in N$ we write
\begin{align}\label{eq:de_epsilond}
  \epsilon_{d} = 
  \begin{cases}
    1, & \textnormal{if } d \equiv 1 \mod 4 , \\
    i, &	\textnormal{if } d \equiv 3 \mod 4 . \\
  \end{cases}
\end{align}

\begin{rem}\label{rem:gpnchih_computed}
  For $\chi = \chi_1$ the trivial character on $(\Z/p^n\Z)^\times$ or
  $\chi = \chi_p = \left( \frac{\cdot}{p} \right)$ the Legendre symbol and $p \neq 2$, we
  find
  \begin{align*}
    g_p\left[ n , \chi_1 , h\right] 
    &= 
      \begin{cases}
        p^{n-1} (p-1), &\textnormal{ if } \nu_p(h) \geq n, \\
        -p^{n-1}, &\textnormal{ if } \nu_p(h) = n-1, \\
        0, &\textnormal{ if } \nu_p(h) < n-1.
      \end{cases}
    \\
    g_p\left[ n , \chi_p , h\right] 
    &= 
      \begin{cases}
        p^{n-1/2} \left( \frac{h/p^{n-1}}{p} \right) \epsilon_p, &\textnormal{ if } \nu_p(h) = n-1, \\
        0, &\textnormal{ if } \nu_p(h) \neq n-1.
      \end{cases}
  \end{align*}
\end{rem}

In order to express the next Gauss sum, we require the following definition.

\begin{de}\label{de:nk_and_R}
  Let $(L,\qq)$ be an integral $\Z$-lattice. By \cite[Lemma 4.3]{SteinHecke} for any $n \in \N$
  and odd prime $p$, there is an orthogonal decomposition of $L/p^nL$ into
  $\Z/p^n\Z$ submodules
  \[
    \left( \bigoplus_i L_i \right) \oplus \left( \bigoplus_j M_j \right) \oplus N,
  \]
  where $L_i = (\Z/p^n\Z) v_i$ is one dimensional with
  $\bilf(v_i,v_i) \in (\Z/p^n\Z)^\times$, $M_j = (\Z/p^n\Z) v_j$ with
  $\bilf(v_j,v_j) \in p^k (\Z/p^n\Z)^\times$ where $1 \leq k \leq n -1$, and
  $\bilf(N,N) \subseteq p^n \Z$. We may assume the $M_j$ being sorted with respect to
  increasing valuation $k$ of $\qq(v_j)$. Define
  \begin{equation}
    n_k = \# \{ 1 \leq i \leq m : \nu_p(\bilf(v_i,v_i)) = k \} \textnormal{ and } R_{p}^{n} \coloneq \sum_{k=0}^{n-1} (n-k) n_k.
  \end{equation}
\end{de}

\begin{rem}\label{rem:GLpnh_as_GLpn1}
  For $p>2$ we find that the following reduction is true (cf. Definition~\ref{de:nk_and_R}):
  \[
    G_{L,p}(n,h) = \left( \frac{h}{p} \right)^{R_{p}^{n}} \underbrace{p^{n R_{p}^{1}} \cdot \prod_{ k=0 }^{n-1} \left( \epsilon_{p^{n - k}} \sqrt{p^{n + k}} \right)^{n_k} \cdot \prod_{i=1}^{n_k} \left( \frac{\qq(v_i)\|\qq(v_i)\|_p}{p^{n} \|\qq(v_i)\|_p} \right)}_{=G_{L,p}(n,1)}.
  \]
\end{rem}

The following Gauss sum is required in Proposition~\ref{prop:VVMF_Heckeops_Fouroercoefficients_even} and in the subsequent
results building upon it.

\begin{de}\label{de:GausssumDfd}
  Let $\Df$ be a discriminant form. For $d \in \N$ define
  \begin{align}
    G_{\Df}(d) = \sum_{\lambda \in \Df} e \left( d \qd(\lambda)\right).
  \end{align}
\end{de}

In \cite{BruinierStein}, the authors compare \cite[Lem.~4.6 p.~115]{McGraw2003} to \cite[Thm.~5.4 p.~329]{Borcherds2000} in
order to conclude that for $\gcd(d,\level(\Df)) = 1$ the following association
defines a character:
\begin{equation}\label{eq:QuotofGauss_char}
  (\Z/\level(\Df)\Z)^\times \ni d \mapsto \frac{G_{\Df}(1)}{G_{\Df}(d)}. 
\end{equation}

The following remark represents groundwork for the proof of Corollary~\ref{cor:Recover_Bruinier_Stein_even}.

\begin{rem}\label{rem:GDdGLp}
  Let $\Df$ be the discriminant form of a lattice $L$ of rank $\dimV$ and
  $d \in \N$. Then, by \cite[Thm.~5.2.2]{Barnard2003} we find
  \[
    G_{\Df}(d) = \sqrt{|\Df|} \sqrt{i}^{\sgn(\Df)} \frac{1}{d^{m/2}} \sum_{\nu \in L/dL} e \left( -\frac{1}{d} \qq(\nu) \right).
  \]
  In particular, the case $d=p^n$ yields
  \[
    G_{\Df}(p^n) = \sqrt{|\Df|} \sqrt{i}^{\sgn(\Df)} \frac{1}{p^{nm/2}} \cdot G_{L,p}(n,1).
  \]
  As a consequence, with $\Dftors{p}$ denoting the $p$-torsion we find for
  $p \nmid \level(L)$ that
  \begin{align}\label{eq:GDdGLp}
    \frac{G_{\Df}(1)}{G_{\Df}(p)} = p^{m/2} \frac{1}{\overline{G_{L,p}(1,1)}} = \frac{p^{-m/2}}{|\Dftors{p}|} G_{L,p}(1,1).
  \end{align}
\end{rem}

\begin{proof}
  Only the last equation has to be proven. By~\cite[Prop.~3.8]{ScheithauerWeilRep} we find the
  identity $|G_{\Df}(d)| = \sqrt{|\Dftors{d}||\Df|}$. As a consequence,
  \[
    |G_{L,p}(n,1)| = \sqrt{|p^{mn}| |\Dftors{p}|}.
  \]
  This immediately implies the result. In case of $p \mid \level(L)$, the
  sum~$G_{\Df}(p^n)$ might vanish.
\end{proof}

Do note that explicit formulae for a more general version of $G_{\Df}$ are also
contained in \cite[Thm.\ 3.9 p.\ 11]{ScheithauerWeilRep}.

	\section{Modular forms}

\begin{de}
	A function $f : \Ha \to \C[L'/L]$ is called \emph{(vector-valued) modular function} of weight $k \in \frac{1}{2} \Z$ if for all $\gamma = (M, \phi) \in \Mp_{2}(\Z)$ the following transformation law holds 
	\[
	f(\gamma \tau) = \phi(\tau)^{2k} \cdot \left[\rho_L(\gamma) f\right](\tau).
	\]
					\end{de}

The operation of $\Telmp$ implies that such a function $f$ possesses a Fourier expansion if it is holomorphic, meaning there are coefficients $a(\lambda,n) \in \C$ such that 
\begin{equation}\label{eq:VVMF_FourierExp}
	f( \tau ) = \sum_{\lambda \in L'/L} \sum_{n \in \qd(\lambda) + \Z} a(\lambda,n, \Im(\tau)) e(n\tau) \efr_{\lambda}.
\end{equation}
In case the discriminant group $L'/L$ is trivial, the index $\lambda \in L'/L$ is eliminated from the notation. 
We say that a holomorphic modular function $f$ is a \emph{modular form}, 
if all coefficients $a(\lambda,n)$ in \eqref{eq:VVMF_FourierExp} vanish for negative $n \in \Q$. 
Further, it is called a \emph{cusp form}, if all coefficients $a(\lambda,n)$ vanish for non-positive $n \in \Q$. 
The spaces of such modular functions are denoted by $\MF_{L,k}$ and $\CF_{L,k}$. 
Given a modular form $f \in \MF_{L,k}$, there is a general procedure to induce a modular form to a sublattice $M \leq L$. 
In that case, $M \leq L \leq L' \leq M'$ implying the inclusion $L'/M \leq M'/M$ and 
there is a lift $\uplift{L}{M} : \MF_{L,k} \to \MF_{M,k}$ declared by 
\begin{align}\label{eq:upop_concrete}
    \left(\uplift{L}{M}f\right)_\mu = 
    \begin{cases}
        f_{\overline{\mu}}, 	& \textnormal{if } \mu \in L'/M, \\ 
        0,						&	\textnormal{otherwise}, 
    \end{cases}
\end{align} 
where $\overline{\mu}$ refers to the image of $\mu \in L'/M$ under the projection to $L'/L$ (cf. \cite[Sec.~4]{Scheithauer2015}).
This construction plays a role in Subsection~\ref{ssec:Lseries:DefandConv} and there are further means to construct induced vector-valued modular forms from 
existing ones described in \cite[Sec.~4]{Scheithauer2015}. 
However, a concrete class which is vital for our purposes is the following.

\begin{de}\label{de:Theta}
	Let $L$ be even positive definite, $\lambda \in L'/L$, and $\tau = u + iv \in \Ha$. 
	Define 
		\begin{align}
			\theta_{L,\lambda}(\tau)
			\coloneq\;
			&	v^{\tfrac{\negsig}{2}} \cdot \sum_{l \in L + \lambda} e\left(\tau \qq(l) \right),  
			\qquad 
			\Theta_{L} \coloneq \sum_{\lambda \in L'/L} \theta_{L,\lambda}(\tau) \efr_{\lambda}.
	\end{align}
\end{de}

These theta series $\Theta_{L}$ are modular forms of weight $\possig/2$ (cf.\ \cite[Thm.~4.1 p.~505]{Borcherds1998}) and their Fourier coefficients equal representation numbers of the lattice $L$. 
Comparing these theta forms to Eisenstein series will yield formulae for the representation numbers in terms of divisor functions up to the coefficients of a cusp form. 
Thus, deriving asymptotic bounds on the Fourier coefficients of cusp forms is vital for quantifying these formulae. 
Additionally, these bounds serve, in our context, as a basis for proving convergence results for $L$-series associated to cusp forms. 
We will reduce this question to the scalar-valued case and begin with stating the following observation.

\begin{prop}[{\cite[Thm.~5.4]{Borcherds2000}}]
	Suppose that $\Df$ is a discriminant form of level dividing $N$. 
	If $b$ and $c$ are divisible by $N$, then $\gamma = \left( \sm{a}{b}{c}{d} , \sqrt{c \tau + d} \right) \in \Mp_{2}(\Z)$ acts on the Weil representation by 
	\[
	\dwrep_{L}(g) \efr_{\lambda} = \chi_{\Df}(g) \efr_{d \lambda}
	\]
	where $\chi_{\Df}$ denotes the character of $\overline{\Gamma}_0(N)$ defined in \cite[Thm.~5.4]{Borcherds2000}.
\end{prop}

In fact, in case of $4 \nmid N$, the character $\chi_{\Df}$ is trivial on the preimage of $\Gamma(N)$ 
in $\Mp_{2}(\Z)$. In the other case, meaning $4 \mid N$, the section $s(\Gamma(N))$ in $\Mp_{2}(\Z)$ is contained in the kernel of the character (cf.\ \cite[Lem.~5.3 p.~329]{Borcherds2000}). 
Hence, the component functions $f_{\lambda}$ of a vector-valued modular form $f \in \MF_{L,k}$ are actually contained in 
$\MF_{k}(\Gamma(\level(L)))$.

As a consequence, for the purpose of asymptotically bounding the coefficients of the vector-valued modular form $f$, 
it suffices to bound 
the asymptotic behaviour of coefficients of scalar-valued modular forms for $\Gamma(N)$. 
There is the following standard bound available in the literature (cf.\ \cite[Lem.~3.62 p.~90]{Shimura1971}, \cite[p.~447]{ShimuraMF_HalfIntegralWeight}).

\begin{lemma}\label{lem:CF_FC_trivialbound}
	Let $k \in \tfrac{1}{2}\Z$ and $N \in \N$ with $4 \mid N$ if $k$ is not integral. 
	Suppose that $f \in \CF_{k}(\Gamma(N))$ has Fourier expansion 
	\[
	f(\tau) = \sum_{n \in \Z} a(n) e(n/N \tau).
	\]
	Then 
	\[
	a(n) \in \LandO(n^{k/2}) \textnormal{ for } n \to \infty.
	\]
\end{lemma}

Also compare \cite{Jin2018} deriving that if the component functions of a vector-valued modular form satisfy 
this bound, then it is a cusp form. 
Combining this bound with the bounds on coefficients of Eisenstein series yields the following result. 

\begin{cor}\label{cor:Hecke_bound}
	Let $f \in \MF_{k}(\Gamma(N))$ with Fourier expansion as above. 
	Then 
	\[
	a(n) \in \LandO(n^{k-1}).
	\]
\end{cor}

In many cases there are sharper bounds on the asymptotic behaviour known, 
for instance, the bound $\LandO(n^{k/2-1/5})$ by Rankin in the even case (cf.~\cite[Thm.~2, p.~358]{Rankin1939II}). 
However, even sharper bounds are to be found in the literature for forms to larger congruence subgroups. 
Consequently, a reduction to the setting of these bigger congruence subgroups is desirable. 
Recall that $\Gamma(\level(L)) \trianglelefteq \Gamma_{1}(\level(L))$ with abelian quotient, 
which remains true under the section $s$ from~\eqref{eq:section_Gamma4_metaplectic} so that we may, by a standard representation theoretic argument, reduce to the case of a modular form for $\Gamma_1(\level(L))$ with finite character.

In a subsequent step, 
this information may be utilised to reduce to the even simpler case of~$\Gamma_{0}(N)$ modular forms. 
However, first note that a scalar-valued modular form~$f$ that has been transferred by the Petersson slash operator 
still possesses a Fourier expansion. 
Hence, it is meaningful to speak of a Fourier expansion of forms that 
have been altered by such a procedure. 
The proof idea of the following lemma is found in \cite[1.2.11 p.~24]{Diamond2005}.

\begin{lemma}\label{lem:fgamma_has_FC}
	Let $\Gamma$ be a congruence subgroup, $k \in \tfrac{1}{2}\Z$, $f \in \MF_{k}(\Gamma)$, and 
	select $\gamma \in \MGLtwop(\Q)$. 
	Then $f \vert_{k} \gamma$ has a Fourier expansion 
	and its constant term vanishes if $f$ is a cusp form. 
\end{lemma}

With this tool at hand, we may reduce to the case of a modular form for $\Gamma_1(N)$ without character 
which, in turn, reduce to modular forms for $\Gamma_{0}(N)$ with character factoring through $\Gamma_{1}(N)$. 
We denote by $\nu_{p}(N)$ for a prime $p$ the $p$-valuation of $N$.

\begin{lemma}
	Let $k \in \tfrac{1}{2}\Z$, $N \in \N$ which is assumed to be a square with $\nu_{2}(N) \in 3 \N$ if $k$ is not integral. 
	Suppose $\psi : \Gamma_{1}(N) \to \Tc$ is a character that is trivial on $\Gamma(N)$, 
	and 
	\[
	f = \sum_{n \in \N/N} a(n) e(n \tau) \in \CF_{k}(\Gamma_{1}(N),\psi). 
	\] 
		Assume there is some $l \in \N$ as well as some $\alpha \in \R$ such that for all characters 
	$\chi : \Gamma_0(N^2) \to \Tc$ that are trivial on $\Gamma_{1}(N^2)$
	and any choice $g = \sum_{n \in \N} b(n) q^n \in \CF_{k}(\Gamma_0(N^2),\chi)$ we have $b(n^l) \in \LandO(n^{\alpha})$.
	Then also 
	\[
		a(n^l) \in \LandO(n^{\alpha}).
	\]
\end{lemma}

\begin{proof}
	Assume $f$ to be as above and $k \notin \Z$ implying $4 \mid N$. 
	The case of $k \in \Z$ is similar but easier to prove.  
	We demonstrate that $f(N \argdot)$ is a modular form with trivial character for $\Gamma_{1}(N^2)$.
	Recall that there is a section $s: \Gamma_{1}(N) \to \mathcal{G}$ given by \eqref{eq:section_Gamma4_metaplectic}. 			Let $\gamma = \left(\abcds, \phi(\tau) \right) \in s\left(\Gamma_{1}(N^2)\right)$, define $M_N \coloneq \left( \sm{N}{0}{0}{1} , 1 \right) \in \MGLtwop(\Q)$, and compute 
	\begin{align*}
		M_{N} \cdot \gamma \cdot M_{N}^{-1} = \left( \m{a}{Nb}{c/N}{d} , \phi(\tau/N) \right) \in s(\Gamma(N)). 	\end{align*}
	The inclusion follows from the condition on $N$ in conjunction with \eqref{eq:section_Gamma4_metaplectic}. 
												Hence, for any choice of $\gamma \in s(\Gamma_{1}(N^2))$ we infer
	\begin{align*}
		f(M_N \gamma \tau) 	
		=	f(M_N \gamma M_N^{-1} M_N \tau) 
		=	\psi(M_N\gamma M_N^{-1}) f(M_N \tau) \phi(\tau)^{2k}. 
	\end{align*}
	Recall that by assumption $\psi$ was trivial on $s(\Gamma(N))$. 
	As a result, we have that 
	\[
		f(N \tau) = f(M_{N} \tau) = \sum_{n \in \N/N} a(n) e(Nn \tau) \in \CF_{k}(\Gamma_{1}(N^2)). 
	\]
	Finally, the space $\CF_{k}(\Gamma_{1}(N^2))$ may be represented as a direct sum of spaces 
	of modular forms for $\Gamma_0(N^2)$ with character 
		so that the coefficients of $f$ must fulfil $a(n^{l}) \in \LandO(n^{\alpha})$ by assumption. 
\end{proof}

As a consequence, it suffices to cite asymptotic bounds of Fourier coefficients for modular forms 
for $\Gamma_{0}(N)$ with nebentypus. 
There are some sharper bounds available in this setting. 
For modular forms of integral weight, there is the Deligne bound, 
providing the best possible bound.

\begin{thm}[Ramanujan-Petersson-Deligne]\label{thm:Deligne_bound}
	For $f \in \CF_k(\Gamma_0(N),\chi)$ with Fourier expansion $\sum_{0 < n \in \Z/N} a(n) e(n\tau)$, 
	we have for $\gcd(n,N) = 1$ and any $\varepsilon > 0$ that 
	\[
	a(n) = \mathcal{O}(n^{k/2-1/2+\varepsilon}), \qquad n \to \infty.
	\] 
\end{thm}

Due to the structure of the Hecke algebra it suffices to find a similar bound for coefficients at primes that divide the level. 
In fact, by Atkin--Lehner theory, it suffices to do so for newforms.
The following is due to Ogg and Li and part of \cite[Thm.~3 p.~295]{Li1974}. 

\begin{prop}
	Let $k \in \Z_{> 0}$ and $f \in \CF_{k}^{\textnormal{new}}(\Gamma_0(N),\chi)$ with Fourier expansion 
	\[
	f(\tau) = \sum_{n = 1}^{\infty} a(n) e(n\tau)
	\]
	be a normalised form that is an eigenform for the Hecke algebra $\Heckealg_{N}$. 
	Then for $p \mid N$ we find the following
	\begin{enumerate}[label=\roman*)]
		\item 
		If $\chi$ is not induced by a character modulo $N/p$, then $\lvert a(p) \rvert = p^{k/2 - 1/2}$. 
		
		\item 
		If $\chi$ is a character modulo $N/p$ and $p^2 \nmid N$, then $a(p)^2 = \chi(p) p^{k-1}$. 
		
		\item 
		If $\chi$ is a character modulo $N/p$ and $p^2 \mid N$, then $a(p) = 0$.  
	\end{enumerate}
\end{prop} 

As a Corollary we obtain the following bound on Fourier coefficients. 

\begin{thm}\label{thm:Deligne_bound_general_Gamma0Nchi}
	Let $k \in \Z_{> 0}$ and $f \in \CF_{k}(\Gamma_0(N),\chi)$ with Fourier expansion
	\[
		f(\tau) = \sum_{n = 1}^{\infty} a(n) e(n\tau).
	\]
	Then 
	\[
		a(n) \in \LandO_{\varepsilon}(n^{k/2-1/2 + \varepsilon}). 
	\]
\end{thm}

However, for forms of half-integral weight, the situation is more delicate. 
Nevertheless, there are the following two bounds available in the literature. 
The first has been proven by Blomer and Harcos and is derived from \cite[Cor.~2 p.~55]{Blomer2008}.

\begin{thm}\label{thm:BlomerBound}
	Let $k \in (\Z/2) \setminus \Z$ with $k \geq 5/2$, $N \in \N$, and $\chi$ be a Dirichlet character modulo $4N$. 
	Let
	\[
	f(\tau) = \sum_{n = 1}^{\infty} a(n) e(n\tau)
	\] 
	be in $\CF_{k}(\Gamma_{0}(4N),\chi)$. Then for $t \in \N$ and $\gcd(n,2N)=1$ we find
	\[
	a(tn) \in \LandO_{\varepsilon}(n^{k/2 - 5/16 + \varepsilon}).
	\]
\end{thm}

Note that the above theorem is also true in case of $k = 3/2$, 
if the choice of cusp form is reduced to the orthogonal complement of theta functions. 
\\
The second bound goes back to Iwaniec \cite{Iwaniew1987} and Duke \cite{Duke1988} in a special case 
and has been generalised 
by Waibel \cite[Thm.~1 p.~186]{Waibel2018}.

\begin{thm}\label{thm:WaibelBound}
	Let $k \in \tfrac{1}{2}\Z \setminus \Z$ with $k \geq 5/2$, $4 \mid N \in \N$, and $\chi$ be a character modulo $4N$. 
	Let
	\[
	f(\tau) = \sum_{n = 1}^{\infty} a(n) e(n\tau)
	\] 
	be in $\CF_{k}(\Gamma_{0}(4N),\chi)$. 
	Then it holds for indices $n = t v^2 w^2$ with $t$ squarefree, $v \mid N^{\infty}$, and $\gcd(w,N) = 1$  
	that	
	\[
	a(n) \in \LandO_{\varepsilon}(n^{k/2 - 1/2 + \varepsilon} v^{1/2}).
	\]
\end{thm}

Note that the above implies $a(n) \in \LandO_{\varepsilon}(n^{k/2-1/4 + \varepsilon})$, 
which is the Weil bound.

In conclusion we have the following bounds on Fourier coefficients of vector-valued modular forms.

\begin{cor}\label{cor:VVMF_bound_FC_for_Lseries}
	Let $2 \leq k \in \tfrac{1}{2}\Z$ and $f \in \CF_{L,k}$ with Fourier expansion
	\[
	f(\tau) = \sum_{\lambda \in \Df} \sum_{n \in \qd(\lambda) + \Z} a(\lambda,n) \cdot e(n \tau) \efr_{\lambda}.
	\]
	Select $t \in \Q$, then the following bound is true for all $\lambda \in \Df$
	\[
	a(\lambda,tn^2) \in 
	\begin{cases}
		\LandO_{\varepsilon} (n^{k - 1 + \varepsilon}), 	& 2 \mid \rk(L), \\
		\LandO_{\varepsilon} (n^{k - 1/2 + \varepsilon}), 	& 2 \nmid \rk(L).
	\end{cases}
	\]
	If in addition $n \in \Q$ is assumed to be coprime to $\level(L)$ we obtain 
	\[
	a(\lambda,tn) \in 
	\begin{cases}
		\LandO_{\varepsilon} (n^{k - 1/2 + \varepsilon}), & 2 \mid \rk(L), \\
		\LandO_{\varepsilon} (n^{k - 5/16 + \varepsilon}), & 2 \nmid \rk(L).
	\end{cases}
	\]
\end{cor}

\subsection{Eisenstein series}

Beyond the theta series presented above, Eisenstein series represent another
important class of modular forms. Our analysis requires the use of
non-holomorphic Eisenstein series, which we introduce briefly based upon \cite[2.5
pp.~30-36]{Kiefer2021}. The statements are essential for our investigations, more
specifically for deriving meromorphic continuation of the symmetric square type
$L$-functions arising from the Rankin--Selberg type integrals in Section~\ref{sec:Lseries}. In
the following, let $(L,\qq)$ be a non-degenerate quadratic lattice and
$k \in \tfrac{1}{2}\Z$ be a half integer and $\overline{\Gamma_{\infty}} \leq \Mp_{2}(\Z)$ be
the group generated by $\Telmp$.

\begin{de}\label{de:VVEisensteinseriesnonholomorphic}
  Let $\lambda \in \Df$ be isotropic and $k \in \tfrac{1}{2}\Z$. Define the \emph{Eisenstein
    series}\index{Eisenstein series!non-holomorphic}
  \begin{equation}
    E_{L,\lambda,k}(\tau,s) \coloneq \frac{1}{2} \sum_{\gamma \in \overline{\Gamma_{\infty}} \setminus \Mp_{2}(\Z)} \left( \Im(\tau)^s \efr_\lambda \right) \vert_{L,k} \gamma .
  \end{equation}
\end{de}

By construction this function transforms as a modular form of weight $k$.

\begin{lemma}\label{lem:ELk_convergence_hypLaplacian_Eigenfunction}
  The Eisenstein series $E_{L,\lambda,k}$ converges normally on $\Ha$ for
  $\Re(s) > 1 - \frac{k}{2}$, is real analytic in $\tau$ and an eigenfunction of the
  hyperbolic Laplace operator of weight $k$ with Eigenvalue $s(s + k -1)$.
\end{lemma}

Recall that the Weil representation $\dwrep_{L}$ acts trivially on
$\Gamma(\level(L))$, so that the components of $E_{L,\lambda,k}(\tau,s)$ are scalar-valued
Eisenstein series for $\Gamma(\level(L))$. These, however, have meromorphic
continuation in $s$. The Fourier expansion of these Eisenstein series is
computed by Bruinier and Kühn but with respect to the dual Weil representation,
resulting in the following slight reformulation.

For $v \in \R^{\times}$ and $k \in \tfrac{1}{2}\Z$ representing the weight in the
respective section, set
\[
  \WhittakerW_s(v) \coloneq \lvert v \rvert^{-k/2} \WhittakerW_{\sgn(v)k/2 ,(1-k)/2 - s} (\lvert v \rvert),
\]
where $\WhittakerW_{\kappa,\mu}(z)$ is the usual Whittaker function as in \cite[(13.14.3)
p.~334]{OlverClark2010}.

\begin{prop}[{\cite[Prop.~3.1 p.~1695]{BruinierKuehn2003}}]\label{prop:VVEisensteinseriesnonholomorphicFourierExpansion}
  For $\sig(\Df) \equiv 2k \mod 2$, the Eisenstein series $E_{L,\lambda,k}$ has the Fourier
  expansion
  \begin{equation}\label{eq:prop:VVEisensteinseriesnonholomorphicFourierExpansion}
    E_{L,\lambda,k}(\tau,s) = \sum_{\mu \in L'/L} \sum_{n \in \Z + \qd(\mu)} c_{\lambda}(\mu,n,s,v) \e(n u) \efr_{\gamma}
  \end{equation}
  where the coefficients $c_{\lambda}(\mu,n,s,v)$ are given by
  \begin{align*}
    \begin{cases}
      (\delta_{\lambda,\mu} + i^{\sig(\Df) + 2k} \delta_{-\lambda,\mu}) v^s + 2 \pi v^{1-k-s} \frac{\Gamma(k+2s-1)}{\Gamma(k+s)\Gamma(s)} \cdot \sum_{x \in \Z \setminus \{0\}} \lvert 2c \rvert^{1-k-2s} H_c(\beta,0,\mu,0), & \textnormal{ if } n = 0, \\
      \frac{2^k \pi^{s+k-1}}{\Gamma(s+k)} \WhittakerW_s(4 \pi n v) \cdot \sum_{c \in \Z \setminus \{0\}} \lvert c \rvert^{1-k-2s} H_{c}(\lambda,0,\mu,n), & \textnormal{ if } n > 0, \\
      \frac{2^k \pi^{s+k-1}}{\Gamma(s)} \WhittakerW_s(4 \pi n v) \cdot \sum_{c \in \Z \setminus \{0\}} \lvert c \rvert^{1-k-2s} H_c(\lambda,0,\mu,n), & \textnormal{ if } n < 0. \\
    \end{cases}
  \end{align*}
  Here, $H_{c}$ denotes the generalised Kloostermann sum\index{Kloostermann!sum}
  \begin{equation}\label{eq:prop:VVEisensteinseriesnonholomorphicFourierExpansion:Kloostermannsum}
    H_{c}(\lambda,m,\mu,n) = \frac{e^{- \pi i \sgn(c)k /2}}{\lvert c \rvert} \sum_{\substack{0 \not\equiv d \mod c \\ \abcds \in \Gamma(1)_{\infty} \backslash \Gamma(1) / \Gamma(1)_\infty}} \overline{\rho}_{\mu,\lambda}\widetilde{\abcd} \e\left( \frac{am+nd}{c} \right),
  \end{equation}
  the matrix coefficients $\rho_{\mu,\lambda}$ are as in Definition~\ref{de:WeilRepCoeffs}.\end{prop}

According to \cite[p. 372]{Hejhal1983} for an isotropic element $\lambda \in \Df$
\begin{align}
  E_{L,\lambda,k}(\tau,s) = \frac{1}{2} \sum_{\mu \in \Iso(L'/L)} c_{\lambda}(\mu,0,s) E_{L,\mu,k}(\tau,1-k-s).
\end{align}

From the Fourier expansion of $E_{L,\lambda,k}$ and the properties of the Whittaker
function $\WhittakerW_s$ (cf.\ \cite[p.~335 (13.14.21)]{OlverClark2010}) follows an asymptotic
bound for the Eisenstein series which is relevant for the application of the
Rankin-Selberg method.

\begin{rem}\label{rem:Eis_locunifbound}
  Let $s \in \C$ such that $E_{L,\lambda,k}(\tau,s)$ is holomorphic. We find for
  $\tau = u + iv \in \Ha$ and $\sigma = \max\{\Re(s), \Re(1-k-s)\}$ that
  \begin{equation}
    E_{L,\lambda,k}(\tau,s) = \LandO(v^{\sigma}), \qquad v \to \infty.
  \end{equation}
  In fact, the constant required to bound may be chosen locally uniformly in
  $s$.
\end{rem}

We also introduce the following notation in order to state a functional
equation:
\begin{equation}\label{eq:de:clambdamu0s}
  c_{\lambda}(\mu,0,s,v) = (\delta_{\lambda,\mu} + i^{\sig(\Df)+2k} \delta_{- \lambda,\mu}) v^{s} + c_{\lambda}(\mu,0,s) v^{1-k-s}. 
\end{equation}

\nocite{Hejhal1976} According to \cite[p. 372]{Hejhal1983} we have the following
transformation property.

\begin{prop}\label{prop:VVEis_funequ}
  Let $\Iso(L'/L)$ denote the isotropic elements of $L'/L$. For $\lambda \in \Iso(L'/L)$
  we find the following functional equation
  \begin{align}
    E_{L,\lambda,k}(\tau,s) = \frac{1}{2} \sum_{\mu \in \Iso(L'/L)} c_{\lambda}(\mu,0,s) E_{L,\mu,k}(\tau,1-k-s).
  \end{align}
\end{prop}

	\section{Hecke theory}\label{sec:Hecketheory}

We write $\Gabar = \Mp_{2}(\Z)$ and recall that for $\alpha \in \MGLtwop(\Q)$, the associated Hecke element is 
\[
	T_{\alpha} = \Gabar \, \alpha \, \Gabar.
\]
Further, the collection of these elements carries a natural structure as a $\Q$-algebra, the so called \emph{Hecke algebra}. 
We are interested in double cosets associated to particular elements. 

For $m \in \N$ define   
\begin{align}\label{eq:def_gm2_metaplecticcase}
	g(m) \coloneq \left( \m{m}{0}{0}{1} , 1 \right) \in \MGLtwop(\Q), \qquad \Hoph(m) \coloneq T_{g(m)}.
\end{align}

In the classical theory, considerations of general Hecke operators are reduced to $\Hoph(p)$ for prime numbers $p$. 
In this spirit, we state the following assertion which is proven by direct computation.

\begin{lemma}\label{lem:Hophrelation_pnmidN}
	We find for $r \in \N$ and a prime $p$ not dividing $\level(L)$ that 
	\[
	\Hoph(p^{r+1}) = \Hoph(p)\Hoph(p^{r}) - p \cdot T_{p \mathcal{I}} \cdot \Hoph(p^{r-1}) - \delta_{r=1} \cdot T_{p \mathcal{I}}. 
	\]
\end{lemma}

Note that this result is consistent with \cite[Lem.~4.5.7 (1) p.~140]{Miyake2006}, where $T(p,p) = T_{p \mathcal{I}}$ and $T(1,p^{e}) = \Hoph(p^{e})$ in our notation. 
Note also that the number $N$ in Miyake may be assumed to equal~$1$ in the current setting. 
By utilising the statement above, we may derive a relation for square indices which is the relevant case in the vector-valued setting.

\begin{lemma}\label{lem:Tpr2_recursion_relation}
	For $2 \leq r \in \N$ and a prime $p$, the following relation is true: 
	\[
	\Hoph(p^{r+2}) 	=  \left[\Hoph(p^{2}) + (1-p) \cdot T_{p \mathcal{I}} \right] \Hoph(p^r) - (1 + \delta_{r=2} p^{-1}) \left( p \cdot T_{p \mathcal{I}} \right)^2 \Hoph(p^{r-2}).
	\]
\end{lemma}

\begin{proof}
	We apply Lemma~\ref{lem:Hophrelation_pnmidN} repeatedly to obtain the following:
	\begin{align*}
		\Hoph(p^{r+2}) 	=\;&	{\Hoph(p)} \Hoph(p^{r+1}) - p \cdot T_{p \mathcal{I}}  \Hoph(p^r)  \\
		=\;&	{\Hoph(p)} \left[ \Hoph(p)\Hoph(p^{r}) - p \cdot T_{p \mathcal{I}}  \Hoph(p^{r-1}) \right] - p \cdot T_{p \mathcal{I}}  \Hoph(p^r)  \\
		=\;&	\left[\Hoph(p^{2}) + (p+1) \cdot T_{p \mathcal{I}} \right] \Hoph(p^r) \\
		&- p \cdot T_{p \mathcal{I}} \left[ \Hoph(p^r) + (p + \delta_{r=2}) \cdot T_{p \mathcal{I}}  \Hoph(p^{r-2}) + \Hoph(p^r) \right] \\
		=\;&	\left[\Hoph(p^{2}) + (1-p) \cdot T_{p \mathcal{I}} \right] \Hoph(p^r) - (1 + \delta_{r=2} p^{-1}) \left( p \cdot T_{p \mathcal{I}} \right)^2 \Hoph(p^{r-2}). 
	\end{align*}
\end{proof}

The Hecke operators introduced above not only constitute an abstract algebra 
but also operate on spaces of modular forms. 
Their action on Fourier coefficients of modular forms facilitates the derivation of product expansions for associated $L$-series, 
which motivates their examination in this work.
The theory for scalar-valued modular forms is well known, 
however, naively extending the explicit construction of Hecke operators for the scalar-valued case to vector-valued modular forms fails, 
due to the Weil representation $\dwrep_{L}$ not being extendable to $\GL_{2}(\Z/\level(L)\Z)$ as a linear action. 
An alternative construction circumventing this issue is presented in \cite{BruinierStein} and is consistent with classical Hecke operators in the scalar-valued case. 
We will briefly sketch the content necessary for our proceedings. 

\begin{de}\label{de:action_deltam2}
	For $m \in \N$ and $\delta(m^2) = \gamma g(m^2) \gamma' \in \Gabar g(m^2) \Gabar$ and $\lambda \in \Df$ set 
	\[
		\efr_{\lambda} \vert_{\Df} g(m^2) \coloneq e_{m \lambda}, \qquad \mathfrak{e}_\lambda \vert_{\Df} \delta(m^2) \coloneq \mathfrak{e}_\lambda \vert_{\Df} \gamma \vert_{\Df} g(m^2) \vert_{\Df} \gamma'. 
	\]
\end{de}

That the notation above is well-defined is a non-trivial fact (cf.~ \cite[Sec.~5]{BruinierStein}). 
The same source established that the following operators define endomorphisms on $\MF_{L,k}$.

\begin{de}\label{de:VV_Heckeops_general}
For representatives $\delta_{i} \in \MGLtwop(\Q)$ such that 
\[
	\Gabar g(m^2) \Gabar = \bigsqcup_{i} \Gabar \delta_{i}
\]
declare the action of the Hecke operator $\Hoph(m^2)$ on $f = \sum_{\lambda \in \Df} f_{\lambda} \efr_{\lambda} \in \MF_{L,k}$ as 
\begin{align}
	f \mapsto f\vert_{L,k} \Hoph(m^2) \coloneq m^{k-2} \sum_{i} \sum_{\lambda \in \Df} (f_{\lambda} \vert_{k} \delta_{i}) \otimes (\efr_{\lambda} \vert_{\Df} \delta_{i}).
\end{align}
\end{de}

Concerning the structure of the generated Hecke algebra we find that 
for \(m,n \in \N\) with $\gcd(m,n) = 1$ it is straightforward from Definition~\ref{de:VV_Heckeops_general} 
and its well definedness to verify the relation~$\Hoph((mn)^2) = \Hoph(m^2) \Hop(n^2)$. 
However, self-adjointness has to be computed. 

Recall the element $g(m^2) = \left(\sm{m^2}{0}{0}{1} , 1 \right) \in \MGLtwop(\Q)$ from~\eqref{eq:def_gm2_metaplecticcase} 
and define the following associated element
\begin{align}\label{eq:de:gd2_ast}
	g^{\ast}(m^2) \coloneq \left( \m{1}{0}{0}{m^2}, m \right) = \Selmp g(m^2) \Selmp^{-1} \in \Gabar g(m^2) \Gabar. 
\end{align}

\begin{lemma}\label{lem:gd2_adjoint_on_discgroup}
	For $\lambda \in \Df$ and $m \in \N$ we find 
	\begin{align}
		\efr_{\lambda} \vert_{\Df} g^{\ast}(m^2) = \sum_{\substack{\sigma \in \Df \\ m \sigma = \lambda}} \efr_{\sigma}
	\end{align}
	and conclude for the standard scalar product on $\C[\Df]$ and $v_1$, $v_2 \in \C[\Df]$ that 
	\begin{align}
		\langle v_1 \vert_{\Df} g(m^2) , v_2\rangle = \langle v_1 , v_2 \vert_{\Df} g^{\ast}(m^2) \rangle, 
	\end{align}
	meaning these operators are adjoint to each other: $\left(\vert_{\Df} g(m^2)\right)^\ast = \vert_{\Df} g^{\ast}(m^2)$. 
\end{lemma}

As a consequence, we obtain the following adjointness on the space of modular forms.

\begin{cor}\label{cor:VVMF_slash_gd2_adjoint}
	Let $f,h \in \CF_{L,k}(\Gabar)$, $m \in \N$ and $g(m^2)$ as above. 
	Then we find for every $\alpha \in \Gabar g(m^2) \Gabar$ that 
	\begin{align}
		\langle f \vert_{L,k} \alpha, h \rangle = \langle f , h \vert_{L,k} \det(\alpha) \alpha^{-1} \rangle. 
	\end{align}
	The expression does not depend on the particular choice of $\alpha$ in the double coset. 
\end{cor}

\begin{proof}
	Recall that 
	\begin{align*}
		\langle f \vert_{L,k} \alpha, h \rangle 
		=	\sum_{\lambda, \mu \in \Df} \langle f_{\lambda} \vert_{k} \alpha , h_{\mu} \rangle \langle \efr_{\lambda} \vert \alpha , \efr_{\mu} \rangle. 
			\end{align*}
	Assume first that $\alpha \in \Gabar$. 	Note that the scalar-valued case of Corollary~\ref{cor:VVMF_slash_gd2_adjoint} is well known for $\Gamma(N)$ modular forms. 
	Recall that the component functions satisfy $f_{\lambda}, h_{\mu} \in \CF_{k}(\Gamma(\level(L)))$ and  
	combine this with the unitaricity of the Weil representation for the desired result. 
	\\
	In the general case, by \mbox{Definition~\ref{de:action_deltam2}} 
	and the invariance of $f$ and $h$ under $\Gabar$, we may reduce to the choice $\alpha = g(m^2)$. 
	This, however is solved by the same argument as above combined with Lemma~\ref{lem:gd2_adjoint_on_discgroup} in the components of the group algebra.
\end{proof}

Combining the results above leads to the following observation.

\begin{thm}[{\cite[Thm.~5.6 p.~269]{BruinierStein}}]\label{thm:Tm2starproperties}
	The operation of $\Hoph(m^2)$ for arbitrary $m \in \N$ determines a linear, 
	cusp form preserving action on the space of modular forms. 
	Further, it is identified to be self-adjoint with respect to the Petersson scalar product 
	and the operators satisfy 
		\[
			\Hoph(m_1^2) \Hoph(m_2^2) = \Hoph(m_1^2m_2^2) 
		\]
	for coprime elements $m_1,m_2 \in \N$. 
\end{thm}

We install the following piece of notation in order to state two corollaries. 

\begin{de}
	For $M \in \N$ 
	let $\Heckealgminimal_{M}^{2}$ denote the Hecke algebra generated by all $\Hoph(p^2)$ 
	and $T_{p \mathcal{I}}$ for primes $p$ with $p \nmid M\level(L)$. 
\end{de}

\begin{cor}\label{cor:Heckealgminimal_properties}
	For any $M \in \N$, the algebra $\Heckealgminimal_{M}^{2}$ acts on $\CF_{L,k}$ as a commutative algebra of 
	cusp form preserving, self-adjoint linear operators. 
\end{cor}

The Corollary follows from Theorem~\ref{thm:Tm2starproperties}. 
In conjunction with Lemma~\ref{lem:Tpr2_recursion_relation} we conclude that for any prime $p$ 
with $p \nmid M\level(L)$, the algebra $\Heckealgminimal_{M}^{2}$ contains $\Hoph(p^{2r})$ for all~$r \in \N$.

Corollary~\ref{cor:Heckealgminimal_properties} immediately implies the existence of simultaneous eigenforms. 

\begin{cor}\label{cor:VVMF_Hoph_simEF}
	Let $M \in \N$ and assume $W \leq \CF_{L,k}$ is a nonzero subspace that is invariant under the 
	operation of~$\Heckealgminimal_{M}^{2}$. 
	Then $W$ has an orthonormal basis consisting of simultaneous eigenforms of~$\Heckealgminimal_{M}^{2}$. 
\end{cor}

Now that the existence of eigenforms is certain, a bound on the respective eigenvalues was advantageous.
An effective method for obtaining such bounds was discovered by Kohnen~\cite{Kohnen1987} 
even though the bound itself, at least for scalar-valued Siegel Modular forms, had been known before (cf.~\cite[Sec.~6]{Weissauer1984}).

\begin{lemma}\label{lem:VVMF_Kohnenbound}
	Let $m \in \N$ and $f \in \CF_{L,k}(\Gabar)$ be a nonzero eigenform of $\Hophk(m^2)$ with eigenvalue $\Hev_{m}$. 
	If $\Dftors{m}$ denotes the $m$ torsion of $\Df$, then 
	\begin{align}
		\lvert \Hev_{m} \rvert < m^{k-2} \lvert \Gabar \backslash \Gabar g(m^2) \Gabar \rvert \cdot \lvert {_{m}\Df} \rvert. 
	\end{align}
	In the particular case where $m = p$ is a prime, we find 
	\begin{align}
		\lvert \Hev_{p} \rvert < p^{k-1} \cdot (p+1) \cdot \lvert {_{p}\Df} \rvert. 
	\end{align}
\end{lemma}

\begin{proof}
	Recall the definition of $g^{\ast}(m^2)$ from~\eqref{eq:de:gd2_ast} 
	and note that for $\lambda \in \Df = L'/L$ we have by Lemma~\ref{lem:gd2_adjoint_on_discgroup} that 
		\[
			\efr_{\lambda} \vert g(m^2) \vert g^\ast(m^2) = \efr_{m \lambda} \vert g^{\ast}(m^2) = \sum_{\substack{\sigma \in \Df \\ m \sigma = m \lambda}} \efr_{\sigma} = \sum_{\sigma \in {_{m}\Df}} \efr_{\lambda + \sigma}. 
		\]
				First, note that Cauchy--Bunyakowsky--Schwarz implies in conjunction with 
	Corollary~\ref{cor:VVMF_slash_gd2_adjoint} that 
	\begin{align*}
		\| f \vert_{L,k} g(m^2)\|_2^2 &\leq \|f \vert_{L,k} g(m^2) g^{\ast}(m^2)\|_2 \cdot \| f\|_2.
	\end{align*} 
	Further, 
	\begin{align*}
		&\|f \vert_{L,k} g(m^2) g^{\ast}(m^2)\|_2^2  \\
		=\;	& \sum_{\lambda, \mu \in \Df} \langle f_{\lambda} \vert_{k} g(m^2) g^{\ast}(m^2), f_{\mu} \vert_{k} g(m^2) g^{\ast}(m^2) \rangle \langle \efr_{\lambda} \vert_{k} g(m^2) g^{\ast}(m^2) , \efr_{\mu} \vert_{k} g(m^2) g^{\ast}(m^2)\rangle \\
		=\; & \sum_{\lambda, \mu \in \Df} \langle f_{\lambda}, f_{\mu} \rangle \lvert {_{m}\Df}\rvert^{2} \langle \efr_{\lambda} , \efr_{\mu} \rangle \\
		=\; & \lvert {_{m}\Df} \rvert^2 \|f\|_2^2.
	\end{align*}
	Here,~${_{m}\Df}$ and~${^{m}\Df}$ denote the $m$-torsion and $m$-multiples in $\Df$, respectively.
	Also compare~\eqref{eq:Df_torsion_mults}. 
						Second, 
	\begin{align*}
		\Hev_{m} \cdot \langle f , f \rangle 	&=	\left\langle \Hoph(m^2) f , f \right\rangle 	\\
		&= m^{k-2} \!\! \sum_{h \in \Gabar \backslash\Gabar g(m^2) \Gabar} \left\langle f\vert_{L,k} h  , f \right\rangle \\
		&= m^{k-2} \cdot \lvert \Gabar \backslash\Gabar g(m^2) \Gabar \rvert \cdot \left\langle f\vert_k g(m^2)  , f \right\rangle,
	\end{align*}
	where we have used, in the last line that by Corollary~\ref{cor:VVMF_slash_gd2_adjoint}, 
	the value of the inner product does not depend on the representative in the double coset. 
	Again, by Cauchy--Bunyakowsky--Schwarz and the above calculations, we infer 	\begin{align}\label{eq:CSHev:lem:VVMF_Kohnenbound}
		\frac{|\Hev_{m}|}{\left|\Gabar \backslash\Gabar g(m^2) \Gabar \right|} 	&\leq	m^{k-2}  \frac{ \|f \vert_{L,k} g(m^2) \|_2 \cdot \|f \|_2}{\|f\|_2^2}  =		m^{k-2} \cdot  \lvert {_{m}\Df} \rvert. 
	\end{align}
	This settles the bound in question with possible equality. 
	In fact, in \eqref{eq:CSHev:lem:VVMF_Kohnenbound}, we have equality if, and only if, $f \vert_{L,k} g(m^2)$ and $f$ are proportional to each other. 
	Assume this was the case and recall the Fourier expansion \eqref{eq:VVMF_FourierExp} of $f$, meaning there was some constant $C \in \C^{\times}$ with  
	\begin{align}\label{eq:multiple:lem:VVMF_Kohnenbound}
		C \cdot \sum_{\lambda \in \Df} \sum_{n \in \qd(\lambda) + \Z} a(\lambda,n) \; e\left(\tau n \right) \otimes \mathfrak{e}_{\lambda} = 
		\sum_{\lambda \in \Df} \sum_{n \in \qd(\lambda) + \Z} a(\lambda,n) \; e\left(\tau nm^2 \right) \otimes \mathfrak{e}_{m\lambda}. 
	\end{align}
		Select an arbitrary index $(\lambda,n)$. 
	Note that \eqref{eq:multiple:lem:VVMF_Kohnenbound} remains, by iteratively applying $\vert_{L,k}g(m^2)$, true, 
	if we replace the pair~\((C,m)\) by \((C^l,m^l)\) for arbitrary $l \in \N$. 
	Evidently, there exists some~$l \in \N$, such that $m^{2l} > n\level(L)$, which implied through coefficient comparison that $a(\lambda,n) = 0$, 
	thus establishing the required contradiction. 
	\\
	Finally, for the choice $m = p$, there are $p(p+1)$ cosets in the class $ \Gabar \backslash\Gabar g(m^2) \Gabar$ (cf.\ \cite[p.~451]{ShimuraMF_HalfIntegralWeight}), finishing the proof. 
\end{proof}

\subsection{Hecke action on Fourier coefficients}
On the level of Fourier coefficients of modular forms Bruinier and Stein provide
the following formulae for the operation of~\(\Hoph(p^2)\).

\begin{prop}[{\cite[Prop.~4.3 p.~258]{BruinierStein}}]\label{prop:VVMF_Heckeops_Fouroercoefficients_even}
  Let $L$ have even signature and $p$ be a prime not dividing $\level(L)$. Let
  $f \in \MF_{L,k}$ and denote the Fourier expansion as in \eqref{eq:VVMF_FourierExp}. Then
  \[
	f\vert_{L,k} \Hophk(p^2) = \sum_{\lambda \in \Df} \sum_{n \in \Z + \qd(\lambda)} b(\lambda,n) q^n \otimes \efr_{\lambda},
  \]
  where
  \begin{align}
    b(\lambda,n) = p^{2k-2} a(\lambda/p,n/p^2) 
    + \frac{G_{\Df}(1)}{G_{\Df}(p)} p^{k-2} (p \delta_{p \mid n} - 1) a(\lambda,n) 
    + a(p \lambda, p^2 n)
  \end{align}
  and
  \[
	\delta_{p \mid n} =
	\begin{cases}
      1, & p \mid n, \\
      0, & p \nmid n.
	\end{cases}
  \]
  Moreover, we understand that $a(\lambda/p,n/p^2) = 0$ if $p^2 \nmid n$.
\end{prop}

Here, for $d \in \N$, the Gauss sum $G_{\Df}(d) = \sum{\lambda \in \Df} e(d \qd(\lambda))$ is
further described in Definition~\ref{de:GausssumDfd} and below. Bruinier and Stein have also
considered the case of odd signature.

\begin{thm}[{\cite[Thm 4.10 p~263]{BruinierStein}}]\label{thm:VVMF_Heckeops_Fouroercoefficients_odd}
  Let $L$ have odd signature and $p$ be a prime not dividing $\level(L)$. Let
  $f \in \MF_{L,k}$ and denote the Fourier expansion as in \eqref{eq:VVMF_FourierExp}. Then
  \[
    f\vert_{L,k} \Hophk(p^2) = \sum_{\lambda \in \Df} \sum_{n \in \Z + \qd(\lambda)} b(\lambda,n) q^n \otimes \efr_{\lambda},
  \]
  where
  \begin{align}
    b(\lambda,n) = p^{2k-2} a(\lambda/p,n/p^2) 
    + \epsilon_{p}^{\sig(\Df) + (\tfrac{-1}{\lvert \Df \rvert})} \left(\frac{p}{\lvert \Df \rvert 2^{\sig(\Df)} }\right) \left(\frac{-n}{p}\right) p^{k-3/2} a(\lambda,n) 
    + a(p \lambda, p^2 n)
  \end{align}
  and for an odd integer $d$ we set
  \[
	\epsilon_{d} =
	\begin{cases}
      1, & d \equiv 1 \mod 4, \\
      i, & d \equiv -1 \mod 4.
	\end{cases}
  \]
  Moreover, we understand that $a(\lambda/p,n/p^2) = 0$ if $p^2 \nmid n$.
\end{thm}

Note that for our purposes a formula that is valid for all primes is required.
An advance in that direction has been made in \cite[Thm.~5.4 p.~246]{SteinHecke} for all odd
primes. However, it has been pointed out by \cite{Bouchard2018} that the cited Theorem is
flawed. The original author has submitted a correction \cite{Stein2021}, representing the
previously arising, but flawed, coefficients in a more abstract fashion in terms
of certain representation numbers of the lattice that also appear as
coefficients of Eisenstein series in~\cite[p.~447]{BruinierKuss2001}.

These representation numbers, however, are not explicit enough to derive product
expansions of $L$-series. Hence, the coefficients have to be calculated more
explicitly based on the exposition \cite{SteinHecke}. To this end we have to let $\Hoph(p^2)$
explicitly act on $f \in \MF_{L,k}$ where the operation is given by letting right
cosets of
\[
  \Gabar g(p^2) \Gabar, \textnormal{ where } g(p^2) = \left( \m{p^2}{0}{0}{1} , 1 \right)
\]
operate as in Definition~\ref{de:VV_Heckeops_general}. Recall that by \cite[p.~263]{BruinierStein} or \cite[p.~451]{ShimuraMF_HalfIntegralWeight} we find
\begin{align}
  \Gabar g(p^2) \Gabar = \Gabar g(p^2) \cup \bigcup_{h (p)^{\ast}} \Gabar \beta_{h} \cup \bigcup_{b (p^2)} \Gabar \gamma_{b}, 
\end{align}
where
\begin{align}\label{eq:betah_gammab}
  \beta_h 
  = 
  \left(
  \left(
  \begin{array}{cc}
    p & h \\
    0 & p \\
  \end{array}
  \right)
  , \sqrt{p} 
  \right),
  \qquad
  \gamma_b
  =
  \left(
  \left(
  \begin{array}{cc}
    1 & b  \\
    0 & p^2 \\
  \end{array}
  \right)
  , p
  \right).
\end{align}
																																	
The operation on the scalar component functions of the vector-valued modular
form via the Petersson slash operator is well known and understood. As a
consequence, it suffices to describe the action of the elements in \eqref{eq:betah_gammab} on the
discriminant group. The action of $\gamma_{b}$ is presented in \cite[Thm.~5.1 p.~243]{SteinHecke}
in greater generality than required.

\begin{prop}\label{prop:Stein_action_gammab}
  Let $p$ be an odd prime, $b \in \Z/p^2\Z$, and $\lambda \in \Df$. Then $\gamma_{b}$ acts in
  the following fashion:
  \begin{align}
    \efr_{\lambda} \vert_{\Df} \gamma_{b} = \sum_{\substack{\nu \in \Df \\ p \nu = \lambda}} e(-b \qd(\nu)) \efr_{\nu}. 
  \end{align}
\end{prop}

The action of $\beta_{h}$, however, will be extracted from \cite[Prop.~5.3 p.~26]{Bouchard2018}
where we select the special case of $l = a = 1$.

\begin{prop}\label{prop:Creutzig_action_betah}
  Let $p$ be an odd prime and $h \in (\Z/p\Z)^{\times}$, write $m = \rk(L)$, and select
  $\lambda \in \Df$. Then
  \begin{equation}\label{eq:prop:Creutzig_action_betah}
    \efr_{\lambda} \vert_{\Df} \beta_{h} = p^{-\dimV/2} \sum_{\substack{\delta \in \Df(p) \\ p \delta = \lambda}} e(-h\qd_{p}(\delta)) \efr_{\lambda},
  \end{equation}
  where $\qd_{p}$ denotes $\overline{p \qq} : \Df(p) = L(p)'/L(p) \to \Q/\Z$.
\end{prop}

This will have to be rewritten for our purposes.

\begin{lemma}\label{lem:betah_action_Df}
  With the notation as above, we find
  \[
    \efr_{\lambda} \vert_{\Df} \beta_{h} = p^{-\dimV/2} \delta_{\lambda \in \Dfmult{p}} \cdot e\left( - hp \qq(\posvec_{\lambda/p})\right) G_{L,p}(1,-h) \efr_{\lambda},
  \]
  where $G_{L,p}(n,h) = \sum_{v \in L/p^nL} e\left( \frac{h}{p^n} \qq(v) \right)$ is
  the Gauss sum defined in Definition~\ref{de:Gauss_sums}, $\posvec_{\lambda/p} \in L'$ is some element,
  such that $p \cdot \overline{\posvec_{\lambda/p}} = \lambda$, $\Dfmult{p}$ are the
  $p$-multiples in $\Df$ and $\delta$ denotes a Kronecker symbol.
\end{lemma}

\begin{proof}
  Beginning with \eqref{eq:prop:Creutzig_action_betah} we transform the sum by utilising the following description
  $\Df(p) = p^{-1}L'/L \simeq L'/pL$:
  \[
	p^{-\dimV/2} \sum_{\substack{\delta \in \Df(p) \\ p \delta = \lambda}} e(-h\qd_{p}(\delta)) \efr_{\lambda} = p^{-\dimV/2} \sum_{\substack{\mu \in L'/pL \\ \mu = \lambda \mod L}} e\left( -\frac{h}{p}\qq(\mu) \right) \efr_{\lambda},
  \]
  where $\qq(\mu)$ is to be understood as the quadratic form on $L'$ of some
  preimage in $L'$ of $\mu \in L'/pL$. In order for the sum above not to vanish, it
  is necessary that $\lambda \in \Dfmult{p}$ (cf.\ \cite[Lemma 5.2.1]{Barnard2003}). Note that by the
  isomorphism $(L'/pL)/(L/pL) \simeq \Df$, the solutions in $\mu$ of $\mu = \lambda \mod L$ are
  represented by $\lambda + L/pL$ so that the expression above equals
  \[
	p^{-\dimV/2} \delta_{\lambda \in \Dfmult{p}} \cdot \sum_{v \in L/pL} e\left( -\frac{h}{p}\qq(\posvec_{\lambda} + v) \right) \efr_{\lambda},
  \]
  where $\posvec_\lambda$ is some lift of $\lambda$ to $L'$. In fact, the sum does not
  depend on the chosen representative, as the shift may be absorbed into $v$. We
  extract the dependency on $\lambda$ from the sum. To this end, let $\nu \in \Df$ be a
  $p$-th root of $\lambda$, i.e.\ $p \cdot \nu = \lambda \in \Df$ and denote a lift of~$\nu$ to~$L'$
  by~$\posvec_{\nu}$. Then
  \begin{align*}
    \delta_{\lambda \in \Dfmult{p}} \sum_{ v \in L/pL } e \left( - \tfrac{h}{p} \qq(\ell_\lambda + v)\right) 
    = \delta_{\lambda \in \Dfmult{p}} \cdot e \left( - \tfrac{h}{p} \qq(p\ell_\nu) \right) \cdot \sum_{ v \in L/pL } e \left( - \tfrac{h}{p} \qq(v) \right)  \tag{$\ast$}.
  \end{align*}
  In fact, since $\lambda \in \Dfmult{p}$ there is $\nu \in \Df$ such that $p \nu = \lambda$. Hence,
  there must be $\ell_\nu \in L'$ and $\delta \in L$ such that $p \ell_\nu + \delta = \ell_\lambda$.
  Nevertheless, the term $\delta \in L$, may be absorbed in~$v$ and will henceforth be
  ignored resulting in replacing $\qq(\ell_\lambda + v)$ by $\qq(p \ell_\nu + v)$. Then,
  however,
  \[
	\qq(p\ell_\nu + v) = \qq(p\ell_\nu) + p \bilf(\ell_\nu, v) + \qq(v) \equiv \qq(p \ell_\nu) + \qq(v) \mod p.
  \]
  This, in fact, means that the reduction above is valid.
\end{proof}

Following the methodology of \cite{SteinHecke}, we derive explicit formulae for the action of
Hecke operators on the Fourier expansion of vector-valued modular forms. In
fact, the following is a modified version of \cite[Thm.~5.4 p.~246]{SteinHecke}.

\begin{prop}\label{prop:VVMF_Heckeops_Fouriercoefficients_bad_primes}
  Let $p$ be an odd prime and $f \in \MF_{L,k}$ with Fourier expansion
  \[
	f(\tau) = \sum_{\lambda \in \Df} \sum_{n \in \qd(\lambda) + \Z} a(\lambda,n) \cdot e(n\tau) \efr_{\lambda}.
  \]
  Then the Fourier coefficients of $f \vert_{L,k} \Hophk(p^2)$ are given by
  \begin{align*}
    b(\lambda,n) =\;	& p^{2(k-1)} \delta_{\lambda \in \Dfmult{p}} \sum_{\substack{\lambda' \in \Dftors{p} \\ n - p^{2} \qd(\lambda/p + \lambda') \in p^{2} \Z }} a\left( \lambda/p + \lambda',\frac{n - p^{2} \qd(\lambda/p + \lambda')}{p^{2}} + \qd(\lambda/p + \lambda') \right) \\
    +\;	&  p^{k-2} \delta_{\lambda \in \Dfmult{p}} K_{L,p}   g_p\left[1,{_-\chi_p}^{R}, n- \qq(p \cdot \posvec_{\lambda/p})\right] \; a\left(\lambda, n \right)  \\
    +\;	&	a(p \lambda,p^{2}n).
  \end{align*}
  Here, $K_{L,p} = p^{-m/2} G_{L,p}(1,1)$ and $g_p[r,\chi,n]$ are the Gauss sums
  from Definition~\ref{de:Gauss_sums}, while $\posvec_{\lambda/p} \in L'$ is any element such that
  $\overline{p \posvec_{p/\lambda}} = \lambda \in \Df$. Further, $\Dftors{p}$ denotes the
  $p$-torsion of $\Df$, $\Dfmult{p}$ the $p$-multiples, and
  $\delta_{\lambda \in \Dfmult{p}} = 1$ if $\lambda \in \Dfmult{p}$ and $0$, otherwise. In addition,
  $_{-}\chi_{p}(h) = \left( \tfrac{-h}{p} \right)$ is the Legendre symbol and the
  exponent $R = R_{p}^1$ has been presented in Definition~\ref{de:nk_and_R}.
\end{prop}

\begin{proof}
  Although the proof structure remains consistent for higher powers of $p^2$, we
  focus specifically on the operation of $\Hophk(p^2)$ to maintain notational
  clarity. We have
  \begin{align}
    f \vert_{k,L} \Hophk(p^{2}) 
    =& \; p^{k-2} \sum_{\lambda \in \Df} f_\lambda \vert_k g(p^2) \; \mathfrak{e}_\lambda \vert_L g(p^2) \label{eq:1:prop:VVMF_Heckeops_Fouriercoefficients_bad_primes} \\
     &+ p^{k-2}\sum_{h \in (\Z/p \Z)^\times} \sum_{\lambda \in \Df} f_\lambda \vert_k \beta_{h} \; \mathfrak{e}_\lambda \vert_L \beta_{h} \label{eq:2:prop:VVMF_Heckeops_Fouriercoefficients_bad_primes}\\
     &+ p^{k-2}\sum_{b \in \Z/p^{2} \Z} \sum_{\lambda \in \Df} f_\lambda \vert_k \gamma_b \; \mathfrak{e}_\lambda \vert_L \gamma_b. \label{eq:3:prop:VVMF_Heckeops_Fouriercoefficients_bad_primes}
  \end{align} 
  For the first part \eqref{eq:1:prop:VVMF_Heckeops_Fouriercoefficients_bad_primes} we find
  \begin{align*}
    &\; p^{k-2} \sum_{\lambda \in \Df} f_\lambda \vert_k g(p^2) \; \mathfrak{e}_\lambda \vert_{\Df} g(p^2) \\
    =	&\; p^{2(k-1)} \sum_{\lambda \in \Df} f_\lambda\left( p^{2} \tau \right)  \; \mathfrak{e}_{p\lambda} \\ 
    =	&\; p^{2(k-1)} \sum_{\lambda \in \Df} \sum_{n \in \Z + \qd(\lambda)} a(\lambda,n) \cdot e\left( p^{2} n\tau \right)  \; \mathfrak{e}_{p\lambda} \\
    =	&\; p^{2(k-1)} \sum_{\lambda \in \Dfmult{p}} \sum_{\lambda' \in \Dftors{p}} \sum_{n \in p^{2}\left( \Z + \qd(\lambda/p + \lambda') \right)} a(\lambda/p + \lambda',n/p^{2}) \cdot e\left(n\tau \right)  \; \mathfrak{e}_{\lambda}.
  \end{align*}
  Here, $\lambda/p$ is any fixed $p$-th root of $\lambda$ in $\Df$. Also compare~\eqref{eq:Df_torsion_mults}. The
  inner sum in the expression above may be rewritten as
  \[
	\sum_{n - p^{2} \qd(\lambda/p + \lambda') \in p^{2} \Z } a\left( \lambda/p + \lambda',\frac{n - p^{2} \qd(\lambda/p + \lambda')}{p^{2}} + \qd(\lambda/p + \lambda') \right) \cdot e\left(n\tau \right) \; \mathfrak{e}_{\lambda}
  \]
  in order to separate the integer part of the index from the rest.
  \\
  Next, we evaluate the sum \eqref{eq:2:prop:VVMF_Heckeops_Fouriercoefficients_bad_primes} involving the action of the elements $\beta_{h}$:
  \begin{align*}
    & \;	p^{k-2} \sum_{h \in (\Z/p \Z)^\times} \sum_{\lambda \in \Df} f_\lambda \vert_k \beta_{h} \; \mathfrak{e}_\lambda \vert_L \beta_{h}.
  \end{align*}
  First, we utilise Lemma~\ref{lem:betah_action_Df} to obtain
  \begin{align*}
    &	\;	p^{k-2} \sum_{h \in (\Z/p \Z)^\times} \sum_{\lambda \in \Df} f_\lambda \vert_k \beta_{h} \; \mathfrak{e}_\lambda \vert_L \beta_{h} \\
    =	&	\;	p^{2(k-1)} \sum_{h \in (\Z/p \Z)^\times} \sum_{\lambda \in \Dfmult{p}} p^{-k} f_\lambda \left( \frac{p \tau + h}{p} \right) \cdot p^{-\tfrac{\dimV}{2}} e\left( - hp \qq(\posvec_{\lambda/p})\right)  G_{L,p}(1,-h)\efr_{\lambda}. \tag{$\ast$}
  \end{align*}
  Here, $\posvec_{\mu}$ for some $\mu \in \Df$ is understood to be an arbitrary but
  fixed lift of $\mu$ to $L'$ and $\lambda/p$ is any fixed $p$-th root of $\lambda$ in $\Df$.
  In $(\ast)$, the dependency of the Gauss sum $G_{L,p}(1,-h)$ on $h$ may be
  extracted. In fact, an application of Remark~\ref{rem:GLpnh_as_GLpn1} yields
  \begin{equation}\label{eq:proof2:prop:VVMF_Heckeops_Fouriercoefficients_bad_primes}
    G_{L,p}(1,-h) = \left( \frac{-h}{p} \right)^{R} \cdot G_{L,p}(1,1)
  \end{equation}
  with $R = R_{p}^{1}$ as in Definition~\ref{de:nk_and_R}. With these computations, and by
  writing~${_{-}\chi_{p}}(h)$ for the Legendre symbol appearing in~\eqref{eq:proof2:prop:VVMF_Heckeops_Fouriercoefficients_bad_primes}, we find the
  following expression for ($\ast$):
  \begin{align*}
    &\; p^{k-2- \dimV/2} G_{L,p}(1,1) \sum_{h \in (\Z/p \Z)^\times} \sum_{\lambda \in \Dfmult{p}} f_\lambda \left( \frac{p \tau + h}{p} \right) \cdot e\left( - hp \qq(\posvec_{\lambda/p})\right) {_{-}\chi_{p}}(h)^{R} \efr_{\lambda} \\
    =&\; p^{k-2- \dimV/2} G_{L,p}(1,1) \sum_{h \in (\Z/p \Z)^\times} \sum_{\lambda \in \Dfmult{p}} \sum_{n \in \qd(\lambda) + \Z} a(\lambda,n) \, e\!\left(n\tau\right) e\left(\tfrac{h}{p} [n - \qq(p \cdot \posvec_{\lambda/p})]\right) {_{-}\chi_{p}}(h)^{R} \efr_{\lambda}. 
  \end{align*}
  Set $K_{L,p} \coloneq p^{-\dimV/2} G_{L,p}(1,1)$ and compare Definition~\ref{de:Gauss_sums} to
  recognise the appearing term above as the following Gauss sum:
  \[
	g_{p}\left[1,{_{-}\chi_p^R}, n - \qq(p \cdot \posvec_{\lambda/p})\right] = \sum_{h \in (\Z/p \Z)^\times} e\left(\tfrac{h}{p} [n - \qq(p \cdot \posvec_{\lambda/p})]\right) \chi_{p}(-h)^{R}
  \]
  which implies the term in ($\ast$) to equal
  \[
	p^{k-2} K_{L,p} \sum_{\lambda \in \Dfmult{p}} \sum_{n \in \qd(\lambda) + \Z} a(\lambda,n) g_{p}\left[1,{_{-}\chi_p^R}, n - \qq(p \cdot \posvec_{\lambda/p})\right] e\left(n\tau\right) \efr_{\lambda}.
  \]
  Last, the sum in \eqref{eq:3:prop:VVMF_Heckeops_Fouriercoefficients_bad_primes} involving the action of $\gamma_b$ is evaluated, using
  Proposition~\ref{prop:Stein_action_gammab}.
  \begin{align*}
    & \;	p^{k-2}\sum_{b \in \Z/p^{2} \Z} \sum_{\lambda \in \mathcal{L}} f_\lambda \vert_k \gamma_b \; \mathfrak{e}_\lambda \vert_L \gamma_b \\
    =	& \;	p^{-2}\sum_{b \in \Z/p^{2} \Z} \sum_{\lambda \in \mathcal{L}} f_\lambda\left( \frac{\tau + b}{p^{2}} \right) \cdot  \sum_{\substack{\nu \in \mathcal{L} \\ p \nu = \lambda}} e\left(- b \qd(\nu) \right) \efr_\nu \\ 
    =	& \;	p^{-2} \sum_{\lambda \in \mathcal{L}} \sum_{\substack{\nu \in \mathcal{L} \\ p \nu = \lambda}} \sum_{n \in \qd(\lambda) + \Z } a(\lambda,n) \sum_{b \in \Z/p^{2} \Z} e \left( b\frac{n - \qq(p\posvec_{\nu})}{p^{2}}\right) e\left( \frac{n\tau}{p^{2}} \right) \efr_\nu.  \tag{$\ast'$}
  \end{align*}
  Note that
  \begin{align*}
    \sum_{b \in \Z/p^{2} \Z} e \left(\frac{n - \qq(p\posvec_{\nu})}{p^{2}}\right)^{b} = p^{2} \delta_{n \equiv \qq(p\posvec_{\nu}) \mod p^{2} \Z}.
  \end{align*}
  Further, $p \nu = \lambda$ may only be fulfilled if $\lambda \in \Dfmult{p}$. The solutions to
  this equation are exactly the elements $\lambda/p + \delta$ where $\delta \in \Dftors{p}$ and
  $\lambda/p$ is a fixed inverse image. However, by~\eqref{eq:Df_torsion_mults}, these are all elements in
  $\Df$. In total, we obtain the following expression for $(\ast')$:
  \begin{align*}
    & \;	\sum_{\lambda \in \Dfmult{p}} \sum_{\substack{\nu \in \mathcal{L} \\ p \nu = \lambda}} \sum_{\substack{n \in \Z + \qd(\lambda) \\ n - \qq(p\posvec_{\nu}) \in p^{2} \Z}} a(\lambda,n)  e\left( \frac{n\tau}{p^{2}} \right)  \mathfrak{e}_\nu \\ 
    =	&	\; \sum_{\mu \in \mathcal{L}} \; \sum_{n - p^2\qq(\posvec_{\mu}) \in p^{2} \Z} a(p \mu,n)  e\left( \frac{n\tau}{p^{2}} \right)  \mathfrak{e}_{\mu} \\
    =	&	\; \sum_{\mu \in \mathcal{L}} \; \sum_{m \in \Z + \qd(\mu)} a(p \mu,p^{2}m)  e\left( m \tau \right)  \mathfrak{e}_{\mu}. \qedhere
  \end{align*}
\end{proof}

As a consequence of the above, we conclude the case of good primes (different
from $2$), which essentially recovers Proposition~\ref{prop:VVMF_Heckeops_Fouroercoefficients_even}.

\begin{cor}[Bruinier, Stein]\label{cor:Recover_Bruinier_Stein_even}
  Let $L$ have even signature, $p$ be an odd prime not dividing $\level(L)$ and
  $f \in \MF_{L,k}$ with Fourier coefficients $a(\lambda,n)$. Denote the Fourier
  coefficients of $f \vert_{L,k} \Hophk(p^2)$ by $b(\lambda,n)$. Then we find
  \begin{align}\label{eq:Recover_Bruinier_Stein_even}
    b(\lambda,n) = p^{2k-2} a(\lambda/p,n/p^2) + \frac{G_\Df(1)}{G_\Df(p)} p^{k-2} (\delta_{p \mid n} p -1) a(\lambda,n) + a(p\lambda,p^2n).
  \end{align}
\end{cor}

\begin{proof}
  In the present case, the formulae of Proposition~\ref{prop:VVMF_Heckeops_Fouriercoefficients_bad_primes} read
  \begin{align*}
    b(\lambda,n) =\;\;	& p^{2(k-1)} a\left( \lambda/p, n/p^2 \right) \\
                    &+ p^{k-2-m/2} G_{L,p}(1,1)   g_p\left[1,{_-\chi_p}^R, n- q(p \cdot \lambda/p)\right] \; a\left(\lambda, n \right)  \\
                    &+ a(p \lambda,p^{2}n).
  \end{align*}
  As a consequence, only the middle term has to be evaluated, meaning the sums
  $g_p$ and $G_{L,p}$ are to be understood. Recall that Equation \eqref{eq:GDdGLp} in Remark~\ref{rem:GDdGLp}
  yields the desired quotient of Gauss sums. Hence, only proving that the second
  Gauss sum $g_p$ reduces to $(\delta_{p \mid n} p - 1)$ is required. Since $R=m$ is
  even in this context, it follows that $\chi_p^{R} = \chi_1$ reduces to the trivial
  character. Thus, Remark~\ref{rem:gpnchih_computed} implies
  \[
	g_p[1,\chi_1,n-\qq(p \cdot \posvec_{\lambda/p})] = \delta_{p \mid n - \qq(p \cdot \posvec_{\lambda/p})} p - 1.
  \]
  In fact, the denominator of the individual fractions does not contain $p$,
  since $p \nmid \level(L)$. Hence, investigating the numerator for divisibility by
  $p$ suffices. On the other hand, $\qq(p \cdot \posvec_{\lambda/p})$ is clearly divisible
  by $p$, so that the question reduces to whether $p \mid n$ or not.
\end{proof}

\begin{rem}
  Proposition~\ref{prop:VVMF_Heckeops_Fouriercoefficients_bad_primes} suggests a potential approach for estbalishing bounds on the
  growth of Fourier coefficients of vector-valued modular forms by using the
  recursion relation for Fourier coefficients of eigenforms. For progressions in
  good primes, the procedure works straightforwardly, for bad primes, however,
  the situation is more delicate.
\end{rem}

	\section{\(L\)-series}\label{sec:Lseries}

\subsection{Definition and convergence}\label{ssec:Lseries:DefandConv}
	
We recall that we fixed an even non-degenerate $\Z$-lattice \((L,\qq)\), and that for a half-integral number $k \in \tfrac{1}{2} \Z$ 
a modular form $f \in \MF_{L,k}$ has a Fourier expansion as in \eqref{eq:VVMF_FourierExp} with Fourier coefficients $a(\lambda,n)$. 
	
\begin{de}\label{de:symmsqLfunforsublattice}
	For $f \in \MF_{L,k}$, a splitting sublattice $M = L_1 \oplus L_2 \leq L$ with definite $L_1 \otimes_{\Z} \R$, 
	an isotropic element $\eta \in (L'\cap L_2')/(L \cap L_2')$, and $t \in \Q^{\times}$ define the formal series 
	\begin{equation}\label{eq:symmsqLfunforsublattice}
		L_{L_1,\eta,t}(f,s) \coloneq \sum_{0 \neq l \in L_1' \cap L'} \frac{a(\overline{l} + \eta, t \qq(l))}{\qq(l)^{s}}.
	\end{equation} 
\end{de}

\begin{lemma}\label{lem:LL1etarfsconverges}
	Assume $L = L_1 \oplus L_2$ splits with definite $L_1 \otimes_{\Z} \R$ of rank $\dimV_1$ and $2 \leq k \in \tfrac{1}{2}\Z$.
	In case $f \in \MF_{L,k}$, the series $L_{L_1,\eta,r}(f,s)$ converges normally for $\Re(s) > k - 1 + \tfrac{m_1}{2}$, 
	rendering it holomorphic.  
	If, in addition, $f \in \CF_{L,k}$, we find that $L_{L_1,\eta,r}(f,s)$ converges normally 
	for $\Re(s) > \frac{k + m_1}{2} - \boundoptimiser$ with $\sigma = 1/2$ in the even case and $\sigma = 1/4$ in the odd case.
\end{lemma}

\begin{proof}
	By assumption $\sqrt{\lvert \qq \rvert} : L_1 \otimes_\Z \R \to \R_{\geq 0}$ defines a norm, 
	rendering it equivalent to~$\| \argdot \|_\infty$. 
	This, together with the bound on Fourier coefficients from Corollary \ref{cor:VVMF_bound_FC_for_Lseries} 
	reduces to bounding
	\[
		\sum_{0 \neq l \in L_1'} \|l\|_{\infty}^{k - 2(\boundoptimiser + \Re(s))} 
	\]
	which is, by a counting argument and comparison to the Riemann $\zeta$-function, identified to converge for 
	$\Re(s) > \frac{k + m_1}{2} - \boundoptimiser$. 
	Normal convergence follows immediately.
\end{proof}
	
	For the purpose of deriving product expansions, we are interested in the case of one dimensional sublattices 
	yielding the following definition taylored to our setting. 
	
\begin{de}\label{de:symmsqLfunVVMF}
						Let $f \in \MF_{L,k}$ be a holomorphic modular form. 
	For $\lambda \in \Df$ and $t \in \Q^{\times}$ we define the associated \index{Lfunction@$L$-function!vector-valued!symmetric square} \emph{symmetric square} $L$-function and its \emph{specialisation} at some $N \in \N$ as 
	\begin{align}
		L_{(\lambda,t)}(f,s) \coloneq \sum_{n \in \N} \frac{a(n\lambda,n^2t)}{n^s}, \qquad L_{(\lambda,t)}^{N}(f,s) \coloneq \sum_{\substack{n \in \N \\ \gcd(n,N) = 1}} \frac{a(n\lambda,n^2t)}{n^s}.
	\end{align}
	Further, for $\posvec \in L'$ with $\qq(\posvec) > 0$, we write \(L_{\posvec}^{N} \coloneq L_{(\lambda,t)}\) for $(\lambda,t) = (\overline{\posvec}, \qq(\posvec))$. 
											\end{de}
	
A recombination of the proof of Lemma~\ref{lem:LL1etarfsconverges} and bounds established in Corollary~\ref{cor:VVMF_bound_FC_for_Lseries} yields the following.
	
	\begin{cor}\label{cor:LellNfsabsoluteconvergence}
		Let $f \in \CF_{L,k}$ with Fourier expansion 
		\[
			f(\tau) = \sum_{\lambda \in \Df} \sum_{n \in \qd(\lambda) + \Z} a(\lambda,n) \cdot e(n \tau) \efr_{\lambda}.
		\]
		Then for $N \in \N$ the series 
		\[
			L_{(\lambda,t)}^{N}(f,s) = \sum_{\substack{n \in \N \\ \gcd(n,N) = 1}} \frac{a ( \lambda,n^2 t)}{n^s}
		\]
		converges absolutely for $\Re(s) > k + 1 - 2\boundoptimiser$ for some $1 / 4 \leq \boundoptimiser \leq 1/2$ as follows.  
		\begin{center}
			\begin{tabular}{c|cc}
				$\boundoptimiser$ 	& $N \in \N$ & $ \level(L) \mid N^{\infty}$ \\ \midrule
				$2 \mid \rk(L)$		& $ 1/2 $	& $1/2$ 					\\
				$2 \nmid \rk(L)$	& $ 1/4 $	& $5/16$
			\end{tabular}
			\vspace{11pt}
		\end{center}
	\end{cor}

	\subsection{Rankin Selberg realisation}\label{ssec:Lseries_RankinSelberg}
	
	In 1939 Rankin \cite{Rankin1939II} considered integrals of automorphic functions multiplied
    by a non-holomorphic Eisenstein series of weight zero over the fundamental
    domain of $\SL_{2}(\Z)$ given by
	\[
      \Fd_{\Gamma(1)} = \{ \tau \in \Ha \mid \lvert \Re(\tau) \rvert \leq 1/2, \lvert \tau \rvert \geq 1\}.
	\]
	These would then, by invariance of the integrand, unfold to an integral over
    the vertical strip
	\[
      \Fd_{\Gamma_\infty} = \{ \tau \in \Ha \mid \lvert \Re(\tau) \rvert \leq 1/2\}
	\]
	which represents a fundamental domain of $\Gamma_{\infty} \leq \SL_{2}(\Z)$ being the
    stabiliser of~$\infty$. The resulting integral would represent an $L$-series with
    analytic continuation and functional equation \cite[Thm.~3 p.~360]{Rankin1939II}. We remark
    that Rankin already considered the case of Hecke congruence subgroups and
    that the method was independently developed by Selberg.
	
	In the following, we will mimic that procedure in order to construct
    $L$-series as in~\eqref{eq:symmsqLfunforsublattice} for cusp forms $f \in \CF_{L,k}$. To this end, we
    consider a pairing of $f$ with a non-holomorphic Eisenstein series which is
    twisted with a theta series.

	\begin{lemma}\label{lem:UnfoldingagainstEisenstein}
      Let $L$ be an even lattice that splits as $L = L_1 \oplus L_2$,
      \begin{itemize}
        \item $f \in \CF_{L,k}$ with Fourier coefficients
              $a(\lambda,n) \in \LandO_{\varepsilon}(n^{\nu_{f}})$ for some $\nu_{f} \in \R$,
			
        \item $g \in \MF_{L_1,k_1}$ with Fourier coefficients
              $b(\lambda,n) \in \LandO_{\varepsilon}(n^{\nu_{g}})$ for some $\nu_{g} \in \R$,
			
        \item and $E_{L,\eta,k_2}(\tau,s)$ with isotropic $\eta \in L_2'/L_2$ be the Eisenstein
              series from Definition~\ref{de:VVEisensteinseriesnonholomorphic} such that $k = k_1 + k_2$.
      \end{itemize}
      Then
      \begin{equation}\label{eq:UnfoldingagainstEisenstein}
        \int_{\Fd_{\Gamma(1)}} \left\langle f , g \otimes E_{L,\eta,k_2}(\argdot,s) \right\rangle \cdot \Im^k \di \mu	
        =	\int_{\Fd_{\Gamma_\infty}} \big\langle f , g \otimes \efr_{\eta} \big\rangle \cdot \Im^{\overline{s}+k}   \di \mu 
      \end{equation}
      provided $\Re(s) > \max\{3 - k + \nu_{f} + \nu_{g}, 1 - k/2\}$. In this range,
      both expressions are holomorphic in $\overline{s}$.
	\end{lemma}
	
	\begin{proof}
      In case the right hand side of~\eqref{eq:UnfoldingagainstEisenstein} converges absolutely, the identity
      follows by definition. Hence, only (locally uniform) absolute convergence
      has to be verified. We write~$a(\lambda,n)$ and~$b(\lambda,n)$ for the Fourier
      coefficients of~$f$ and~$g$ as in~\eqref{eq:VVMF_FourierExp} and insert their Fourier series.
      Then
      \begin{align}
        \big\langle f , g \otimes \efr_{\eta} \big\rangle = 
        &\sum_{\lambda_1 \in L_1'/L_1} \sum_{n \in \Z + \qd(\lambda_1)} 
          \sum_{n_1 \in \Z + \qd(\lambda_1)} a(\lambda_1  + \isoel,n) \overline{b(\lambda_1,n_1)} 
          \cdot e(n\tau - n_1 \overline{\tau}).
      \end{align}
      Subsequently, it suffices to consider the summand of an arbitrary, but
      fixed, $\lambda_{1} \in L_1'/L_1$ which we will call $S(\lambda_{1})$. We note that
      \(\lvert e(n\tau - n_1 \overline{\tau}) \rvert = \exp[- 2 \pi v(n + n_1) ]\) and
      conclude \begin{align*} \lvert S(\lambda_{1}) \rvert
        \leq\; & \sum_{\substack{0 < n \in \Z + \qd(\lambda_1) + \qd(\eta) \\0 \leq n_1 \in \Z + \qd(\lambda_1)}} n^{\nu_{f}} n_1^{\nu_{g}} \cdot \exp[- 2 \pi v(n + n_1) ].  \\
        \intertext{We write $N \coloneq \level(L)$ and note that a summation
        over more possible choices of $n,n_1$ potentially increases the
        sum:} \lvert S(\lambda_{1}) \rvert \leq\; & \sum_{\substack{0 \neq n \in \Z \\ n_1 \in \Z}} \frac{\lvert n\rvert^{\nu_{f}} \lvert n_1\rvert^{\nu_{g}}}{N^{(\nu_{f}+\nu_{g})}} \cdot \exp[- 2 \pi v(\lvert n\rvert + \lvert n_1\rvert)/N].
      \end{align*}
      Further, bounding
      \( \lvert n\rvert^{\nu_{f}} \lvert n_1\rvert^{\nu_{g}} \leq \max\{\lvert n\rvert, \lvert n_1\rvert\}^{(\nu_{f}+\nu_{g})} \)
      amounts to
      \begin{align*}
        \int_{\Fd_{\Gamma_\infty}} \lvert S(\lambda_1) \Im^{\overline{s}+k} \rvert \di \mu
        &\leq  \int_{0}^{\infty} \sum_{\substack{0 \neq n \in \Z \\ n_1 \in \Z}} \frac{\max\{\lvert n\rvert, \lvert n_1\rvert\}^{(\nu_{f}+\nu_{g})}}{N^{\nu_{f}+\nu_{g}}} \cdot e^{- \frac{2 \pi}{N} (\lvert n\rvert + \lvert n_1\rvert) \cdot v} v^{\Re(s) + k - 2}  \di v. 
      \end{align*}
      The substitution $v \mapsto v \cdot N/[2 \pi (\lvert n\rvert + \lvert n_1\rvert)]$
      yields that the last expression equals
      \[
        \frac{1}{N^{\nu_{f}+\nu_{g}}} \frac{\Gamma(\Re(s) + k -1)}{(2 \pi/N)^{\Re(s) + k -1}} \cdot \sum_{\substack{0 \neq n \in \Z \\ n_1 \in \Z}} \frac{\max\{\lvert n\rvert, \lvert n_1\rvert\}^{\nu_{f}+\nu_{g}}}{(\lvert n\rvert + \lvert n_1\rvert)^{\Re(s) + k -1}}.
      \]
      By the equivalence of norms, convergence follows by comparison with the
      Riemann $\zeta$-function, yielding locally uniform absolute convergence for
      $\Re(s) > 3 - k + \nu_{f} + \nu_{g}$.
	\end{proof}
	
	Note that the lemma above is also true, if the Eisenstein series is to be
    replaced by a parabolic Poincaré series where the part of the integral for
    $v \to \infty$ may be bounded independently of the polynomial growth of its kernel
    function by a standard argument.
	
	In order to construct the $L$-series presented in \eqref{eq:symmsqLfunforsublattice}, we will have to
    choose $g$ to equal a theta series. Assume that
    $(L,\qq) = (L_1, \qq_1) \oplus (L_2,\qq_2)$ splits with $(L_1,\qq_1)$ positive
    definite of rank~$\dimV_1$ and associated theta function
	\[
      \Theta_{L_1}(\tau) = \sum_{\lambda_1 \in L_1'/L_1} \sum_{l \in L_1 + \lambda_1} e(\qq(l)\tau)\efr_{\lambda_1}.
	\]
	By Lemma~\ref{lem:UnfoldingagainstEisenstein} this yields for some cusp form $f \in \CF_{L,k}$ that
	\begin{align*}
      \int_{\Fd_{\Gamma(1)}} \left\langle f , \Theta_{L_1} \otimes E_{L_2,\isoel,k_2}(\argdot,s) \right\rangle \cdot \Im^k \di \mu  
      = \int_{\Fd_{\Gamma_\infty}} \big\langle f ,  \Theta_{L_1} \otimes \mathfrak{e}_{\isoel} \big\rangle \cdot \Im^{\overline{s}+k}   \di \mu.
	\end{align*}
	Denoting by $a(\lambda,n)$ the Fourier coefficients of $f$ as in~\eqref{eq:VVMF_FourierExp}, we obtain
    through an argument analogous to the proof of Lemma~\ref{lem:UnfoldingagainstEisenstein} that
	\begin{align*}
      &\big\langle f ,  \Theta_{L_1} \otimes \mathfrak{e}_{\isoel} \big\rangle \\
      = 	&\sum_{\lambda_1 \in L_1'/L_1} \sum_{0 < n \in \Z + \qd(\lambda_1)} \sum_{l \in L_1 + \lambda_1} a(\lambda_1  + \isoel,n) \cdot e(n\tau) e\left[ - \overline{\tau} \qq(l) \right]  \\
      = 	&\sum_{\lambda_1 \in L_1'/L_1} \sum_{0 < n \in \Z + \qd(\lambda_1)} \sum_{l \in L_1 + \lambda_1} a(\lambda_1 + \isoel ,n) \cdot e[iv (n + \qq(l))] e[ u (n - \qq(l)) ].
	\end{align*}
	Clearly, $n - \qq(l) \in \Z$, so that the mapping
    $\R \ni u \mapsto e\left[u (n - \qq(l)) \right]$ defines a character which descends
    to a character on the unit circle $\mathbb{T}$. This implies
	\[
      \int_{0}^{1} e\left[u (n - \qq(l)) \right] \di u = \delta_{n,\qq(l)},
	\]
	where the right hand side is to be understood as a Kronecker symbol.

	From these considerations, we deduce that
	\begin{align*}
      &\int_{\Fd_{\Gamma_\infty}} \big\langle f ,  \Theta_{L_1} \otimes \mathfrak{e}_{\isoel} \big\rangle \cdot \Im^{\overline{s}+k}   \di \mu \\
      =\; &\int_{0}^{\infty} \sum_{\lambda_1 \in L_1'/L_1} \sum_{0 \neq l \in L_1 + \lambda_1} a(\lambda_1 + \isoel,\qq(l)) \cdot e[iv 2\qq(l)] \cdot v^{\overline{s}+k-2}   \di v \\
      =\;	& \frac{\Gamma(\overline{s}+k-1)}{(4 \pi)^{\overline{s}+k-1}} \sum_{0 \neq l \in L_1'} \frac{a(\overline{l} + \isoel,\qq(l))}{\qq(l)^{\overline{s}+k-1}}. 
    \end{align*}
    Comparing the last result with Definition~\ref{de:symmsqLfunforsublattice}, we notice that the series
    appearing is exactly $L_{L_1,\eta,1}(f,\overline{s} + k -1)$. Comparison with
    Lemma~\ref{lem:LL1etarfsconverges} yields that the last expression converges absolutely for
    $\Re(s) > \frac{m_1 - k}{2} + 1 - \boundoptimiser$ with $\boundoptimiser$ as
    in the lemma and defines a holomorphic function in $\overline{s}$ in that
    right half plane. As a consequence, we have proven the following
    Proposition.

	\begin{prop}\label{prop:RankinSelbergVVMFThetaposdefsplitcrossEisenstein}
      Let $f \in \CF_{L,k}$, $k \geq 2$, $L = L_1 \oplus L_2$ with positive definite $L_1$
      of rank $\dimV_1$, $E_{L_2,\isoel,k_2}(\argdot,s)$ be the Eisenstein
      series of Definition~\ref{de:VVEisensteinseriesnonholomorphic}, and assume $k = \dimV_1/2 + k_2$. Then we have
      the identity
      \begin{equation}\label{eq:RankinSelbergVVMFThetaposdefsplitcrossEisenstein}
        \int_{\Fd_{\Gamma(1)}} \left\langle f , \Theta_{L_1} \otimes E_{L_2,\isoel,k_2}(\argdot,s) \right\rangle \cdot \Im^k \di \mu	
        =	\frac{\Gamma(\overline{s}+k-1)}{(4 \pi)^{\overline{s}+k-1}} \cdot L_{L_1,\isoel,1}(f,\overline{s} + k - 1)
      \end{equation}
      of holomorphic functions in $\overline{s}$ for
      $\Re(s) > \frac{ m_1 - k }{2} + 1 - \boundoptimiser$, where
      $\boundoptimiser = 1/2$ or $1/4$ depending on whether $\rk(L)$ is even or
      not.
	\end{prop}

	We note that the left hand side of~\eqref{eq:RankinSelbergVVMFThetaposdefsplitcrossEisenstein} gives rise to a meromorphic
    continuation of the $L$-series~$L_{L_{1},\eta,1}$ which is inherited from the
    Eisenstein series. Also compare Remark~\ref{rem:Eis_locunifbound}. Using the notation of \eqref{eq:de:clambdamu0s} as
    well as Proposition~\ref{prop:VVEis_funequ} we deduce the following symmetry.

	\begin{thm}\label{prop:mercont_funequ}
      With the notation as above $L_{L_{1},\eta,1}(f,s)$ has meromorphic
      continuation to the entire complex plane and we find that
      \begin{equation}\label{eq:}
        L_{L_{1},\eta,1}(f,s) = \frac{1}{2} \sum_{\mu \in \Iso(L_{2}'/L_{2})} c_{\eta}(\mu,0,s) L_{L_{1},\mu,1}(f,1-k-s).
      \end{equation}
	\end{thm}

    \subsection{Product expansions}
	The formulae for Fourier coefficients of recursive nature which had been
    induced by a Hecke action in Section~\ref{sec:Hecketheory} allow for deriving product
    expansions of the $L$-series defined above. Recall the definition of~$\epsilon_{d}$
    from~\eqref{eq:de_epsilond} and for~$p$ an odd prime and~$\Df$ a discriminant form, set
	\begin{equation}
      \chi_{\Df}(p) \coloneq \epsilon_p^{\sig (\Df) + \left( \tfrac{-1}{|\Df|}\right)} \left( \frac{p}{|\Df| 2^{\sig(\Df)}} \right). 
	\end{equation}

	\begin{prop}\label{prop:VVMF_Lseries_eigenform_extract_factor} 
      Let $N \in \N$ and $p$ be a prime with $p \nmid \level(L) \cdot N$. Assume
      $f \in \CF_{\Df,k}$ is an eigenform of $\Hoph(p^2)$ with eigenvalue
      $\Hev_{p}$. Then for $\lambda \in \Df$ and $0 < t \in \Z + \qd(\lambda)$ such that
      $p^2 \nmid \level(L)t$ we find that
      \begin{align*}
        L_{(\lambda,t)}^{N}(f,s) 
        =\;	& 	L_{(\lambda,t)}^{Np}(f,s) \\
			& \cdot 
              \begin{cases} 
				\frac{1 + \delta_{p \nmid t} \frac{G_{\Df}(1)}{G_{\Df}(p)} \cdot p^{k-1-s} }{ 1 - \left( \frac{\Hev_{p}}{p^{k-1}} - (1-p^{-1}) \frac{G_{\Df}(1)}{G_{\Df}(p)} \right)  \cdot  p^{k-1-s} + p^{2(k-1-s)} } 	&,	2 \mid \rk(L), \\
				\frac{1 - \chi_{\Df}(p) \left( \frac{-t}{p} \right) p^{-1/2} \cdot p^{k-1-s} }{ 1 - \frac{\Hev_{p}}{p^{k-1}} \cdot p^{k-1-s} + p^{2(k-1-s)}} 	&,	2 \nmid \rk(L).
              \end{cases}
      \end{align*} 
      Given $L_{(\lambda,t)}^{N}(f,\argdot)$ is nonzero, the rational function in
      $p^{-s}$ is well defined and does not vanish in the right half plane of
      absolute convergence of~$L_{(\lambda,t)}^{N}(f,\argdot)$ determined by
      Corollary~\ref{cor:LellNfsabsoluteconvergence}.
	\end{prop}

	\begin{proof}
      Recall that the Fourier coefficients of $f \in \CF_{\Df,k}$ are denoted
      $a(\lambda,n)$ as in \eqref{eq:VVMF_FourierExp}. Following \cite[p.~452]{ShimuraMF_HalfIntegralWeight}, define for $n \in \N$ the formal
      power series
      \[
		H_n(x) \coloneq \sum_{m=0}^\infty a(p^m n \lambda,(p^mn)^2 t) x^m.
      \]
      We seek to prove that there is a rational function $Q(x)$ such that
      \begin{align*}
        H_{n}(x) = a(n \lambda, n^2 t) \cdot  Q(x) 	\tag{$\ast$}. 
      \end{align*}
      Then,
      \begin{align*}
        L_{(\lambda,t)}^{N}(f,s) 
        =\;&	\sum_{\substack{n=1 \\ \gcd(n,N)}}^\infty a(n \lambda , n^2 t) n^{-s} \\
        =\;&	\sum_{\substack{n=1 \\ \gcd(n,Np)}}^\infty H_{n}(p^{-s}) n^{-s} 
        =\;		L_{(\lambda,t)}^{Np}(f,s) \cdot Q(p^{-s}).
      \end{align*}
      We will compute $Q(x)$ in the case of even and odd rank separately.
		
      \textbf{The case of even rank}: In this case, Proposition~\ref{prop:VVMF_Heckeops_Fouroercoefficients_even} yields
      \[
		\Hev_{p} \cdot a(\lambda,n) = a(p\lambda ,p^2 n) + \frac{G_{\Df}(1)}{G_{\Df}(p)} p^{k-2} \left(p \delta_{p \mid n} -1 \right) a(\lambda,tn) + p^{2(k-1)} a(\lambda/p,n/p^2)
      \]
      resulting for $n , m \in \N$ with $p \nmid n$ and $p^2 \nmid \level(L)t$ in
      \begin{align*}
        \Hev_{p} \cdot a(n\lambda , n^2 t) 				=\; &a(pn\lambda ,p^2 n^2t) + \frac{G_{\Df}(1)}{G_{\Df}(p)} p^{k-2} \left(p \delta_{p \mid t} -1 \right) a(n \lambda, n^2t), \tag{I}\\
        \Hev_{p} \cdot a(p^mn \lambda, p^{2m}n^2 t) 	=\; &a(p^{m+1}n\lambda, p^{2(m+1)}n^2 t)  \\
                                                    &+ (p-1) \frac{G_{\Df}(1)}{G_{\Df}(p)} p^{k-2}  a(p^mn\lambda,p^{2m}n^2 t) \\
                                                    & +	 p^{2(k-1)} a(p^{m-1}n\lambda,p^{2(m-1)}n^2t). \tag{II}
      \end{align*}
      Multiplying these equations by the formal variable $x^{m+1}$ and summing
      them results in
      \begin{align*}
        \Hev_{p} x \cdot H_n(x) 
        =\;	&	H_n(x) - a(n \lambda, n^2 t) 
              + (p-1) \frac{G_{\Df}(1)}{G_{\Df}(p)} p^{k-2}  x H_n(x) \\
            &-	   \delta_{p \nmid t} \frac{G_{\Df}(1)}{G_{\Df}(p)} p^{k-1} a(n \lambda, n^2 t) x
              +	p^{2(k-1)} x^2 H_n(x). 
      \end{align*}
      A straightforward rearrangement yields
      \begin{align*}
        H_n(x) = 	&\; a(n \lambda, tn^2) \cdot \left( 1 + \delta_{p \nmid t} \frac{G_{\Df}(1)}{G_{\Df}(p)} p^{k-1}  x\right) \\
        \cdot 			&\left[ 1 - \Hev_{p} x + (p-1) \frac{G_{\Df}(1)}{G_{\Df}(p)} p^{k-2}  x + p^{2(k-1)} x^2 \right]^{-1}.
      \end{align*}
      \textbf{The case of odd rank}: In this case, Theorem~\ref{thm:VVMF_Heckeops_Fouroercoefficients_odd} yields
      \[
		\Hev_{p} \cdot a(\lambda,n) = a(p\lambda ,p^2 n) + \frac{\chi_{\Df}(p)}{\sqrt{p}} \left( \frac{-n}{p} \right) p^{k-1} a(\lambda,n) + p^{2(k-1)} a(\lambda/p,n/p^2)
      \]
      where we have abbreviated
      $\chi_{\Df}(p) \coloneq \varepsilon_p^{\sig (A) + \left( \tfrac{-1}{|A|}\right)} \left( \frac{p}{|A| 2^{\sig(A)}} \right) $.
      This results for $p \nmid n$, $m \in \N$ and $p^2 \nmid \level(L) t$ in
      \begin{align*}
        \Hev_{p} \cdot a(n\lambda ,  n^2 t) 				&= a(pn\lambda , p^2 n^2 t) + \frac{\chi_{\Df}(p)}{\sqrt{p}} \left( \frac{-t}{p} \right) p^{k-1} a(n \lambda, n^2 t), \tag{I}\\
        \Hev_{p} \cdot a(p^mn \lambda, p^{2m}n^2 t) 	&= a(p^{m+1}n\lambda,p^{2(m+1)}n^2 t) +	 p^{2(k-1)} a(p^{m-1}n\lambda,p^{2(m-1)}n^2t). \tag{II}
      \end{align*}
      Analogously to the situation of even rank we find
      \begin{align*}
        \Hev_{p} x \cdot H_n(x) 
        =\;	&	H_n(x) - a(n \lambda, n^2 t) 
              + \frac{\chi_{\Df}(p)}{\sqrt{p}} \left( \frac{-t}{p} \right) p^{k-1}  a(n \lambda, n^2 t) x \\
			&+	 p^{2(k-1)} x^2 H_n(x)
      \end{align*}
      resulting in
      \[
		H_n(x) = a(n \lambda, n^2 t) \left( 1 - \frac{\chi_{\Df}(p)}{\sqrt{p}} \left( \frac{-t}{p} \right) p^{k-1} x\right) \cdot \left[ 1 - \Hev_{p} x + p^{2(k-1)} x^2 \right]^{-1}.
      \]
      This settles the shape of the rational factor.
      \\
      Next, we will investigate the behaviour of the numerator and denominator
      of the rational factor. Write
      \[
		Q(x) = \frac{R(x,t)}{P(x)}
      \]
      for the rational function in $x$. Now assume that
      $L_{(\lambda,t)}^{N}(f,\argdot) \neq 0$, for otherwise there is no behaviour to
      consider. We may fix an $n \in \N$ with $\gcd(n,N) = 1$ such that
      $a(n \lambda, n^2 t) \neq 0$ and assume without loss of generality that $p \nmid n$.
      Then we find for $s$ in the right halfplane of convergence of
      $L_{(\lambda,t)}^{N}(f,\argdot)$ that
      \begin{align*}
        P(p^{-s}) \cdot H_n(p^{-s}) = a(n \lambda, n^{2} t) \cdot R(p^{-s},t). 		\end{align*}
      Observe that unless $R(p^{-s},t) = 0$, no factor in the above equation may vanish. 
      In fact, in case of $p \mid t$, we find $R(p^{-s},t) = 1$ regardless of $s$. 
      On the other hand, if $p \nmid t$, all factors in the second term of $R(p^{-s},t)$ that are not a power of $p$ have absolute value $1$. 
      Hence, we find that 
      \[
		R(p^{-s},t) \neq 0 \textnormal{ for }
		\begin{cases}
          \Re(s) > k - 1,  & 2 \mid \rk(L), \\
          \Re(s) > k - 3/2 & 2 \nmid \rk(L).
		\end{cases}
      \]
      These bounds are strictly sharper than the bounds of absolute convergence
      for $L_{(\lambda,t)}^{N}(f,s)$ imposed on $\Re(s)$ in Corollary~\ref{cor:LellNfsabsoluteconvergence}.
    \end{proof}

	\begin{de}\label{de:infiniteproduct_absoluteconvergence}
      Let $(a_{n})_{n \in \N} \in \C$ be a sequence of numbers. Then, the infinite
      product
      \[
		\prod_{n \in \N} (a_{n})
      \]
      is said to be \emph{absolutely convergent},\index{infinite product!absolute convergence}\index{convergence!absolute!infinite product} if the following series
      converges:
      \[
		\sum_{n \in \N} \, \lvert a_{n} - 1 \rvert.
      \]
	\end{de}

	A key property of absolutely convergent products is that such a product
    attains the value zero, if, and only if, at least one of its factors
    vanishes. We refer the reader to \cite[IV p.~200]{Freitag1} for more on that matter and
    proceed by applying Proposition~\ref{prop:VVMF_Lseries_eigenform_extract_factor} in order to derive infinite product
    expansions.
	
	\begin{cor}\label{cor:Lseries_full_product_good}
      Let $2 \leq k \in \tfrac{1}{2}\Z$ and assume $f \in \CF_{\Df,k}$ is a
      simultaneous eigenform of $\Hophk(p^2)$ with eigenvalue $\Hev_{p}$ for all
      primes $p \nmid \level(L)$. Denote the Fourier expansion of $f$ by
      \begin{align*}
        f(\tau) = \sum_{\mu \in \Df} \sum_{n \in \qd(\lambda) + \Z} a(\mu, n) \cdot e(n\tau) \efr_{\mu} 
      \end{align*}
      and select $\lambda \in \Df$ as well as $t \in \qd(\lambda) + \Z$ with $p^2 \nmid \level(L)t$
      for all $p \nmid \level(L)$. Then
      \begin{align*}
        L_{(\lambda,t)}(f,s) 
        =\;	& 	\left( \sum_{\substack{n \in \N \\ n \mid \level(L)^{\infty}}} \frac{a(n \lambda, n t)}{n^{s}} \right)  
        \\
			& \; \cdot \prod_{p \nmid \level(L)}
              \begin{cases} 
				\frac{1 + \delta_{p \nmid t} \frac{G_{\Df}(1)}{G_{\Df}(p)} \cdot p^{k-1-s} }{ 1 - \left( \frac{\Hev_{p}}{p^{k-1}} - (1-p^{-1}) \frac{G_{\Df}(1)}{G_{\Df}(p)} \right)  \cdot  p^{k-1-s} + p^{2(k-1-s)} } 	&,	2 \mid \rk(L), \\
				\frac{1 - \chi_{\Df}(p) \left( \frac{-t}{p} \right) p^{-1/2} \cdot p^{k-1-s} }{ 1 - \frac{\Hev_{p}}{p^{k-1}} \cdot p^{k-1-s} + p^{2(k-1-s)}} 	&,	2 \nmid \rk(L) 
              \end{cases}
      \end{align*} 
      is absolutely convergent as an infinite product for $\Re(s) > k + 1$ and
      nonzero.
	\end{cor}
	
	\begin{proof}
      The existence of the product is a consequence of Proposition~\ref{prop:VVMF_Lseries_eigenform_extract_factor} and a
      standard argument.
      \\
      In order to conclude absolute convergence of the product, the individual
      rational factors we denote by $Q_p(p^{-s},t)$ have to converge fast enough
      to $1$ for $p \to \infty$ and all choices of $s \in \C$ with $\Re(s) > k + 1$ (cf.\
      Definition~\ref{de:infiniteproduct_absoluteconvergence}). Introduce the following notation for the numerator and
      denominator
      \[
		Q_{p}(p^{-s},t) = \frac{R_p(p^{-s},t)}{P_p(p^{-s})}, \quad \textnormal{so that} \quad Q_{p}(p^{-s},t) - 1 = \frac{R_p(p^{-s},t) - P_p(p^{-s})}{P_p(p^{-s})}.
      \]
      If we prove that $P_{p}(p^{-s})$ converges to $1$, then it suffices to
      verify that $R_p(p^{-s},t) - P_p(p^{-s})$ converges to $0$ faster than
      $1/p$. As before, we distinguish the odd and even case, however, before
      specialising to one of these, recall that by Lemma~\ref{lem:VVMF_Kohnenbound} the following bound
      is true:
      \begin{align}\label{eq:Kohnen_bound:cor:EF_inf_product}
        \lvert \Hev_{p} \rvert < p^{k-1} (p+1).
      \end{align}
      \textbf{The Case of even signature:} We begin by bounding
      \begin{align*}
        \left\lvert P_{p}(p^{-s}) - 1 \right\rvert \leq\; 
        &\left\lvert \frac{\Hev_{p}}{p^{k-1}} - (1-p^{-1}) \frac{G_{\Df}(1)}{G_{\Df}(p)} \right\rvert p^{k-1-\Re(s)} + p^{2(k-1-\Re(s))}  \\ 
        =\;& \left( \frac{p+1}{p} + \frac{1 - p^{-1}}{p} \right) \cdot p^{k-\Re(s)} + p^{2(k-1-\Re(s))}.
      \end{align*}
      We see that for $\Re(s) > k$, the expression converges to $0$ for $p \to \infty$ as
      desired. Next, the numerator is investigated:
      \begin{align*}
        R_p(p^{-s},t) - P_{p}(p^{-s}) 
        =\;	& \left[ \frac{\Hev_{p}}{p^{k-1}} + \left( \delta_{p \nmid t} - (1-p^{-1}) \right)  \frac{G_{\Df}(1)}{G_{\Df}(p)} \right] \cdot p^{k-1-s} - p^{2(k-1-s)}.
      \end{align*}
      We recall that the quotient of Gauss sums defines a character \eqref{eq:QuotofGauss_char} in order
      to verify that the term in brackets may be absolutely bounded by
      \[
		\left\lvert \frac{\Hev_{p}}{p^{k-1}} + \left( \delta_{p \nmid t} - (1-p^{-1}) \right) \frac{G_{\Df}(1)}{G_{\Df}(p)} \right\rvert \leq (p+1) + (1-p^{-1/2}) = \left( 1 + 2 p^{-1} - p^{-2} \right) \cdot p.
      \]
      As a consequence, the difference $R_p(p^{-s},t) - P_{p}(p^{-s})$ goes to
      zero for $p \to \infty$ in case of~$\Re(s) > k$. It goes to $0$ faster than $1/p$
      for $\Re(s) > k + 1$. This settles absolute convergence of the product in
      the even case.
		
      \textbf{The case of odd signature:} We begin by bounding
      \begin{align*}
        \lvert P_{p}(p^{-s}) -1 \rvert \leq\; 
        &\left\lvert \frac{\Hev_{p}}{p^{k-1}} \right\rvert \cdot p^{k-1-\Re(s)} + p^{2(k-1-\Re(s))}  \\
        =\;& \frac{p+1}{p} \cdot p^{k-\Re(s)} + p^{2(k-1-\Re(s))}.
      \end{align*}
      We see that for $\Re(s) > k $, the expression converges to $0$ for $p \to \infty$.
      Next, the numerator is investigated:
      \begin{align*}
        R_p(p^{-s},t) - P_{p}(p^{-s}) 
        =\;	&\left[ \frac{\Hev_{p}}{p^{k-1}} - \chi_{\Df}(p) \left( \frac{-t}{p} \right) p^{-1/2} \right] \cdot p^{k-1-s} - p^{2(k-1-s)}. 
      \end{align*}
      The term in brackets may be bounded by
      \[
		\left\lvert \frac{\Hev_{p}}{p^{k-1}} - \chi_{\Df}(p) \left( \frac{-t}{p} \right) p^{-1/2} \right\rvert \leq (p+1) + \sqrt{p}^{-1} = [1 + p^{-1} + p^{-3/2}] \cdot p.
      \]
      As a consequence, the difference $Q_p(p^{-s},t) - P_{p}(p^{-s})$ goes to
      zero for $p \to \infty$ in case $\Re(s) > k$. It goes to $0$ faster than $1/p$ for
      $\Re(s) > k + 1$. This settles absolute convergence of the product in the
      odd case.
      \\
      Recall that the ranges of absolute convergence for $L_{(\lambda,t)}^{N}(f,s)$
      established in Corollary~\ref{cor:LellNfsabsoluteconvergence} are less restrictive than $\Re(s) > k + 1$.
	\end{proof}

	We proceed to investigate the case of primes $p$ dividing $\level(L)$.
    Beforehand, recall that the associated discriminant form $\Df$ decomposes as
    an orthogonal direct sum of $p$-subgroups $\Df = \oplus_{p} \Df$ and for
    $\lambda \in \Df$ and a prime $p$ we write $\lambda_{p}$ for its $p$-component in
    $\Df_{p}$.

	\begin{prop}\label{prop:VVMF_Lseries_eigenform_extract_factor_alloddprimes}
      Let $f \in \CF_{L,k}$ with Fourier expansion
      \[
		f(\tau) = \sum_{\mu \in \Df} \sum_{n \in \qd(\mu) + \Z} a(\mu,n) \cdot e(n\tau) \efr_{\mu}.
      \]
      Assume $p \neq 2$ is a prime such that $f$ is an eigenform of $\Hophk(p^2)$
      with eigenvalue $\Hev_{p}$ and that there is no nonzero
      $\lambda' \in \mathcal{L}_p$ with $\qd(\lambda') = 0$. Select $\lambda \in \Df$ and
      $t \in \qd(\lambda) + \Z$ such that if $p^2 \mid \level(L)t$ we have that $\lambda_p$ is
      nonzero and~$f$ is invariant under $\OG(\Df_{p})$. Then we find for a
      natural number $N \in \N$ with $p \nmid N$ that
      \begin{align*}
        L_{(\lambda,t)}^{N}(f,s) 
        =\; & L_{(\lambda,t)}^{Np}(f,s) \\
			& \cdot 
              \frac{1 + K_{L,p} 
              \begin{cases}
                \delta_{\lambda_{p} \neq 0} (1 - p^{-1}) + \delta_{\lambda_{p} = 0} \delta_{p \nmid t}  , & 2 \mid R_{p} \\ 
                - \delta_{\lambda_{p} = 0} p^{-1/2} \left(\frac{-t}{p}\right) \epsilon_{p}, & 2 \nmid R_{p}
              \end{cases} 
              \cdot p^{k-1-s} + C(\lambda_{p}) p^{2(k-1-s)}}
              {1 - \left( \frac{\Hev_{p}}{p^{k-1}} - \delta_{2 \mid R_{p}} (1-p^{-1}) K_{L,p} \right) p^{k-1-s} + p^{2(k-1-s)}}.
      \end{align*} 
      If $L_{(\lambda,t)^{N}}(f,\argdot) \neq 0$, the rational factor in $p^{k-1-s}$
      extracted is well defined and nonzero for $\Re(s) > \rk(\Df_{p})/2 + k - 1$.
      Here,
      $K_{L,p} = p^{- \rk(L)/2} \sum_{\nu \in L/pL} e\left( \tfrac{1}{p} \qq(\nu) \right)$
      and $R_{p}$ has been given in Definition~\ref{de:nk_and_R}. Further, the constant
      $C(\lambda_{p}) + 1$ equals the size of the orbit of $\lambda_{p} \in \Df_{p}$ under
      $\OG(\Df_{p})$. In particular, $C(\lambda_{p}) = 0$, if $\lambda_{p} = 0$.
	\end{prop}
	
	In fact, sharper bounds for the nonvanishing of the rational factor may be
    derived.
	
	\begin{proof}
      The proof follows the structure of the proof of Proposition~\ref{prop:VVMF_Lseries_eigenform_extract_factor}, but
      requires additional technical considerations. Along these lines define for
      $n \in \N$ the series
      \(H_n(x) \coloneq \sum_{m=0}^\infty a(p^m n \lambda,(p^mn)^2 t) x^m\). We will prove
      that there is a rational function $Q(x)$ such that
      \begin{align*}
        H_{n}(x) = a(n \lambda, n^2 t) \cdot  Q(x). \tag{$\ast$}
      \end{align*}
      Recall that by Proposition~\ref{prop:VVMF_Heckeops_Fouriercoefficients_bad_primes}, the operation of $\Hophk(p^2)$ on the
      Fourier coefficients of~$f$ reads
      \begin{align*}
        \Hev_{p} a(\lambda,n) =\;	& p^{2(k-1)} \delta_{\lambda \in \Dfmult{p}} \hspace{-.5em} \sum_{\substack{\lambda' \in \Dftors{p} \\ n - p^{2} \qd(\lambda/p + \lambda') \in p^{2} \Z }} \hspace{-.5em} a\left( \lambda/p + \lambda',\frac{n - p^{2} \qd(\lambda/p + \lambda')}{p^{2}} + \qd(\lambda/p + \lambda') \right) \\
        +\;	&  p^{k-2} \delta_{\lambda \in \Dfmult{p}} K_{L,p} g_p\left[1,{_-\chi_p}^{R_{p}}, n - \qq(p \cdot \posvec_{\lambda/p})\right] \; a\left(\lambda, n \right)  \\
        +\;	&	a(p \lambda,p^{2}n).
      \end{align*}
      This results for some $n,m \in \N$ with $p \nmid n$ and $p^2 \nmid \level(L)t$ or
      $\lambda_{p} \neq 0$ in\footnote{ Note that the group $\Df_{p}$ may only contain constituents of the form $\Z/p\Z$. }
      \begin{align*}
        \Hev_{p} a(n\lambda,n^2t) 
        =\;	&  p^{k-2}\delta_{\lambda \in \Dfmult{p}} K_{L,p}   g_p\left[1,{_-\chi_p}^{R_{p}}, n^2 ( t - \qq(p \cdot \posvec_{\lambda/p}))\right] \; a\left(n\lambda, n^2t \right)  \\
        +\;	&	a(pn \lambda,(pn)^{2}t) \tag{I}
      \end{align*}
      and
      \begin{align*}
        \Hev_{p} a(p^mn\lambda,(p^mn)^2t) 
        =\;	& p^{2(k-1)} \sum_{\substack{\lambda' \in \Dftors{p} \\ (p^{m}n)^2t - p^{2} \qd(p^mn\lambda/p + \lambda') \in p^{2} \Z }} a\left( (p^{m}n\lambda)/p + \lambda' , (p^mn)^2 t/p^2 \right) \\
        +\;	&  p^{k-2} K_{L,p}   g_p\left[1,{_-\chi_p}^{R_{p}},  (p^mn)^2 t - \qq(p \cdot \posvec_{(p^mn\lambda)/p}))\right] a\!\left(p^mn\lambda, (p^mn)^2 t\right)  \\
        +\;	&	a(p^{m+1}n \lambda,(p^{m+1}n)^{2}t). \tag{II}
      \end{align*}
      First, the Gauss sum $g_p$ given in Definition~\ref{de:Gauss_sums} and appearing above
      will be expressed more explicitly. We begin by understanding the
      expression in (II) meaning the case of $m \in \N$ being different from zero.
      Note that for a fixed choice $\posvec_{\lambda} \in L'$ projecting to $\lambda \in \Df$ we
      have $t = \qq(\posvec_{\lambda}) + r$ for some integer $r \in \Z$. The expression
      $g_{p}$ in (II) above, however, is independent of that residue $r$ as it
      is multiplied by $p^m$. Further, recall that by Proposition~\ref{prop:VVMF_Heckeops_Fouriercoefficients_bad_primes}, the
      choice of representative $\posvec_{(p^mn\lambda)/p}$ does not matter for the
      value of the Gauss sum \(g_p\), meaning we may choose
      $p^{m-1}n \posvec_{\lambda}$. As a consequence, the Gauss sum in (II) collapses
      to
      \[
		g_p\left[1,{_-\chi_p}^{R_{p}}, (p^mn)^2 t - \qq(p \cdot \posvec_{(p^mn\lambda)/p}))\right] = g_p\left[1,{_-\chi_p}^{R_{p}}, 0 \right].
      \]
      It should be noted that this Gauss sum is computed in Remark~\ref{rem:gpnchih_computed}. In fact,
      we find that
      \[
		g_p\left[1,{_-\chi_p}^{R_{p}}, 0 \right] =
		\begin{cases}
          0 , & \textnormal{ if } 2 \nmid {R_{p}},\\
          p-1, & \textnormal{ if } 2 \mid {R_{p}}.
		\end{cases}
      \]
      Next, the case $m = 0$, corresponding to (I), is considered. Here,
      Remark~\ref{rem:gpnchih_computed} implies
      \[
		g_p\left[1,{_-\chi_p}^{R_{p}}, t - \qq(p \cdot \posvec_{\lambda/p})\right] =
		\begin{cases}
          \delta_{p \nmid t - \qq(p \posvec_{\lambda/p}) } p^{1/2} \left( \frac{ \qq(p \cdot \posvec_{\lambda/p}) - t}{p} \right) \epsilon_{p}, & \textnormal{ if } 2 \nmid {R_{p}},\\
          \delta_{p \mid t - \qq(p \posvec_{\lambda/p})} p-1, & \textnormal{ if } 2 \mid {R_{p}}.
		\end{cases}
      \]
      We deduce that the Gauss sum is indifferent to multiplication of the last
      argument by a square that is coprime to $p$. Multiplication by
      $(\level(L)/p^{\nu_p(\level(L))})^2$ yields that the number
      $(\level(L)/p^{\nu_p(\level(L))})^2p^2 \qq(\posvec_{\lambda/p})$ is an integer
      divisible by $p$.\footnote{Recall that we have assumed $\Df_{p}$ to be non-isotropic, implying $p^2 \nmid \level(L)$.} Hence, $\delta_{p \nmid t - \qq(p \posvec_{\lambda/p})} = \delta_{p \nmid t}$
      and this expression is already implicitly included in the Legendre symbol,
      yielding
      \[
		g_p\left[1,{_-\chi_p}^{R_{p}}, t - \qq(p \cdot \posvec_{\lambda/p})\right] =
		\begin{cases}
          p^{1/2} \left( \frac{- t}{p} \right) \epsilon_{p}, & \textnormal{ if } 2 \nmid {R_{p}},\\
          \delta_{p \mid t} p-1, & \textnormal{ if } 2 \mid {R_{p}}.
		\end{cases}
      \]
      In the next step, we investigate the index of the sum in (II). Further
      analysis of the discriminant group yields that
      $\Df_{p} = \Dftors{p}_{p} \simeq \Dftors{p}$ equals the $p$-torsion of the
      discriminant group $\Df$. As a consequence, any $p$-multiple
      $\mu \in \Dfmult{p}$ has trivial $p$-component $\mu_p$, meaning multiplication
      by $p$ acts as an automorphism on $\Dfmult{p}$, yielding
      \[
        \{p^{m}n \lambda/p + \lambda' \mid \lambda' \in \Dftors{p}\} = \{p^{m-1}n \lambda + \lambda' \mid \lambda' \in \Dftors{p}\}.
      \]
      In addition, recall that $\Dfmult{p}$ is the orthogonal complement of
      $\Dftors{p}$, so that for $\lambda' \in \Dftors{p}$
      \begin{equation}\label{eq:proof_bad_product_qdseparate}
        \qd(p^{m-1}n \lambda + \lambda') = \qd(p^{m-1}n\lambda) + \qd(\lambda')
      \end{equation}
      in case of $m > 1$ or $\lambda_p = 0$. Recall that $t \in \qd(\lambda) + \Z$ so that
      $(p^mn)^2 t - p^{2}\qq(p^{m-1}n \posvec_{\lambda}) \in p^{2} \Z$. In fact, we have
      assumed that there is no isotropic element in $\Df_{p}$, meaning we find
      that the index of the sum in (II) collapses to the case $\lambda' = 0 \in \Df$.
      \\
      Should we have $\lambda_{p} \neq 0$, only the case $m = 1$ remains. Then the
      separation in \eqref{eq:proof_bad_product_qdseparate} works when replacing $\lambda$ by $\lambda - \lambda_{p}$, shifting the
      non-trivial $p$-component to $\lambda'$ this time. Note that there will be at
      least two choices of $\lambda' \in \Dftors{p}$ then. Namely, $\lambda_{p} \neq 0$ and
      $- \lambda_{p}$. In fact, there might be even more and the number of these
      choices $\lambda' \in \Df_{p}$ will be denoted by $C(\lambda_{p}) + 1$. To advance the
      computation, the different Fourier coefficients appearing in the sum in
      (II) have to be related.\footnote{ Note that the sum in equation (II) introduces indices corresponding to various choices of $\lambda'$ that were not present in the original series $L_{(\lambda,t)}^{N}(f,s)$. } This is where the assumption of invariance
      with respect to the action of the orthogonal group comes into play. Recall
      that for $\lambda_{p}\neq 0$, we find $\qd(\lambda_{p}) \neq 0$ by assumption on the
      lattice. This implies the maximality of the local lattice $L_{p}$ which
      renders $\Df_{p}$ an $\F_{p}$ vector space. Then, mapping $\lambda_{p}$ to a
      different choice of root with the same norm defines an isometry of
      subspaces, which, by Witts extension theorem, extends to an isometry of
      $\Df_{p}$, i.e.\ an element in $\OG(\Df_{p})$. By assumption, the Fourier
      coefficients of $f$ are invariant with respect to this action so that we
      obtain $C(\lambda_{p}) + 1$ contributions from this particular coefficient and
      identify $C(\lambda_{p}) + 1$ as the size of the orbit of $\lambda_{p}$ under
      $\OG(\Df_{p})$.
      \\
      With these manipulations, (I) and (II) become
      \begin{align*}
        \Hev_{p} a(n\lambda,n^2t) 
        =\;	&  p^{k-2} \delta_{\lambda \in \Dfmult{p}} K_{L,p}   \left[\delta_{2 \mid {R_{p}}} (\delta_{p \mid t} p - 1 ) + \delta_{2 \nmid {R_{p}}} p^{1/2} \left(\frac{-t}{p}\right) \epsilon_{p} \right] \; a\left(n\lambda, n^2t \right)  \\
        +\;	&	a(pn \lambda,(pn)^{2}t) \tag{I'}
      \end{align*}
      and
      \begin{align*}
        \Hev_{p} a(p^mn\lambda,(p^mn)^2t) 
        =\;	& p^{2(k-1)} [\delta_{m > 1} + \delta_{m = 1} (C(\lambda_{p}) + 1)] a\left( p^{m-1}n \lambda , \frac{(p^mn)^2 t}{p^2} \right) \\
        +\;	& \delta_{2 \mid {R_{p}}} p^{k-2} K_{L,p} \cdot (p-1) \; a\left(p^mn\lambda, (p^mn)^2 t\right)  \\
        +\;	&	a(p^{m+1}n \lambda,(p^{m+1}n)^{2}t). \tag{II'}
      \end{align*}
      We multiply the equations above by $x^{m+1}$ and sum them to obtain
      \begin{align*}
        &\Hev_{p} \cdot H_{n}(x) x \\
        =\; & p^{2(k-1)} H_{n}(x) x^{2} + C(\lambda_{p}) p^{2(k-1)} a(n \lambda , n^2 t) x^2 \\
        &+ \delta_{2 \mid {R_{p}}} p^{k-2} K_{L,p} \left[ (p-1) H_{n}(x)x - a(n\lambda,n^2 t)x \left\{ (p-1) - \delta_{\lambda_{p} = 0} (\delta_{p \mid t}p - 1) \right\}\right] \\
        &+ \delta_{2 \nmid {R_{p}}} p^{k-2} K_{L,p}  \delta_{\lambda_{p} = 0} p^{1/2} \left(\frac{-t}{p}\right) \epsilon_{p} a(n\lambda,n^2t)x \\
        &+ H_n(x) - a(n \lambda, n^2 t). 
      \end{align*} 
      This will have to be rearranged. Beforehand, observe that the term in
      curly braces admits the following reformulation
      \begin{align*}
        (p-1) - \delta_{\lambda_{p} = 0} (\delta_{p \mid t}p - 1) 
        =\;	& (p-1) - \delta_{\lambda_{p} = 0} [(p - 1 + (\delta_{p \mid t}p - p)] \\
        =\;	& \delta_{\lambda_{p} \neq 0} (p - 1) - \delta_{\lambda_{p} = 0} p (\delta_{p \mid t} - 1) \\
        =\;	& \delta_{\lambda_{p} \neq 0} (p - 1) + \delta_{\lambda_{p} = 0} \delta_{p \nmid t} p. 
      \end{align*}
      With that transformation, we may rearrange the above equation for
      $H_{n}(x)$ to read
      \begin{align*}
        & H_{n}(x) - H_{n}(x) x\left( \Hev_{p} - \delta_{2 \mid {R_{p}}} (p-1) p^{k-2} K_{L,p}  \right) + H_{n}(x) x^{2} p^{2(k-1)} \\
        =\;	& a(n\lambda, n^2 t) \times  \left[1 + p^{k-2} K_{L,p} 
              \begin{cases}
				\delta_{\lambda_{p} \neq 0} (p - 1) + \delta_{\lambda_{p} = 0} \delta_{p \nmid t} p , & 2 \mid {R_{p}} \\ 
				- \delta_{\lambda_{p} = 0} p^{1/2} \left(\frac{-t}{p}\right) \epsilon_{p}, & 2 \nmid {R_{p}}
              \end{cases} 
              \cdot x - C(\lambda_{p}) p^{2(k-1)} x^2 \right] 
      \end{align*}
      which yields the desired rational expression, once $x$ is substituted with
      $p^{-s}$:
      \begin{align*}
        &H_{n}(p^{-s})  \\
        =\; 
        &a(n\lambda, n^2 t) \cdot \frac{1 + K_{L,p}
          \begin{cases}
            \delta_{\lambda_{p} \neq 0} (1 - p^{-1}) + \delta_{\lambda_{p} = 0} \delta_{p \nmid t}  , & 2 \mid {R_{p}} \\ 
            - \delta_{\lambda_{p} = 0} p^{-1/2} \left(\frac{-t}{p}\right) \epsilon_{p}, & 2 \nmid {R_{p}}
          \end{cases} 
          \cdot p^{k-1-s} - C(\lambda_{p}) p^{2(k-1-s)}}
          {1 - \left( \Hev_{p} - \delta_{2 \mid {R_{p}}} (1-p^{-1}) p^{k-1} K_{L,p} \right) p^{-s} + p^{2(k-1-s)}}.
      \end{align*}
      It remains to establish the well-definedness of the potential denominator.
      We assumed $L_{(\lambda,t)}^{N}(f,\argdot) \neq 0$ for there is nothing to show
      otherwise. Then we may fix $n$ such that $a(n\lambda,n^2t) \neq 0$ and write
      \[
		Q(p^{-s}) = \frac{R(p^{-s},t)}{P(p^{-s})}
      \]
      for the numerator and denominator of the rational expression above. In the
      right half plane of absolute convergence of $L_{(\lambda,t)}^{N}(f,s)$, also
      $H_{n}(p^{-s})$ converges absolutely, yielding
      \[
		P(p^{-s}) H_{n}(p^{-s}) = a(n\lambda,n^2 t) \cdot R(p^{-s},t)
      \]
      in this right half plane. Note that in the above equation no factor
      vanishes, unless $R(p^{-s},t)$ does. To this end, we determine a right
      half plane in which $\vert 1 - R(p^{-s},t) \rvert < 1$. The first factor
      that needs to be tended to is $K_{L,p}$. By~\cite[Prop.~3.8]{ScheithauerWeilRep} we find
      \[
        \lvert K_{L,p} \rvert = p^{-\rk(L)/2} \lvert G_{L,p}(1,1) \rvert \leq \sqrt{\lvert \Dftors{p} \rvert}.
      \]
      Recall that in our case the local lattice $L_{p}$ was assumed to be
      maximal, implying $\Dftors{p} \simeq \Df_{p}$ which results in its size being
      equal to $p^{\rk(\Df_{p})}$. Also note that $C(\lambda_{p}) = 0$ in case
      $\lambda_{p} = 0$. If $\lambda_{p} \neq 0$, we necessarily find $p \mid \level(L)$ and
      conclude that
      $C(\lambda_{p}) \leq \lvert \Df_{p} \rvert - 2 = p^{\rk(\Df_{p})} - 2$. We first
      examine the event of $2 \mid {R_{p}}$ which yields
      \begin{align*}
        \lvert 1 - R(p^{-s},t) \rvert \leq 
        \begin{cases}
          \left( p^{\rk(\Df_{p})/2} (1 - p^{-1}) + [p^{\rk(\Df_{p})} - 2] p^{k-1-\Re(s)} \right) p^{k-1-\Re(s)}, 	& \lambda_p \neq 0,\\
          p^{\rk(\Df_{p})/2+k-1-\Re(s)}, 	& \lambda_p = 0.
        \end{cases}
      \end{align*}
      Next, the case $2 \nmid {R_{p}}$ is regarded resulting in
      \begin{align*}
        \lvert 1 - R(p^{-s},t) \rvert \leq 
        \begin{cases}
          [p^{\rk(\Df_{p})} - 2] p^{2(k-1-\Re(s))}, 	& \lambda_p \neq 0, \\
          p^{\rk(\Df_{p})/2+ k-3/2-\Re(s)}, 	& \lambda_p = 0.
        \end{cases}
      \end{align*}
      As a consequence, nonvanishing of \(R(p^{-s},t)\) is guarenteed for
      $\Re(s) > \rk_{p}(\Df)/2 + k - 1$.
    \end{proof}

	Synthesising the preceding results yields a complete product expansion.

	\begin{thm}\label{cor:Lell_complete_product}
      Assume $L$ to be maximal and select a simultaneous eigenform
      $f \in \CF_{L,k}$ of all $\Hophk(p^2)$ with eigenvalue $\Hev_{p}$. Write
      \[
		f(\tau) = \sum_{\mu \in \Df} \sum_{n \in \qd(\mu) + \Z} a(\mu,n) \cdot e(n\tau)\efr_{\mu}
      \]
      for its Fourier expansion and select an index $(\lambda,t)$ such that for all
      primes $p$ we have that if $p^2 \mid \level(L)t$, then $\lambda_{p} \neq 0$ and~$f$ is
      invariant with respect to the action of $\OG(\Df_{p})$. We have
      \begin{align*}
        &L_{(\lambda,t)}(f,s)  \\
        \overset{\textnormal{def}}{=}\; 	
        & \sum_{n \in \N} \frac{a(n\lambda,n^2t)}{n^s} \\
        =\;	& \begin{cases}
          \sum_{n \in \N} \frac{a(2^n\lambda,2^nt)}{2^{ns}}, 	& 2 \mid \level(L), \\
          a(\lambda,t),										& 2 \nmid \level(L)
        \end{cases}  \\
        &\cdot \prod_{p \nmid \level(L)}
          \begin{cases} 
            \frac{1 + \delta_{p \nmid t} \frac{G_{\Df}(1)}{G_{\Df}(p)} \cdot p^{k-1-s} }{ 1 - \left( \frac{\Hev_{p}}{p^{k-1}} - (1-p^{-1}) \frac{G_{\Df}(1)}{G_{\Df}(p)} \right)  \cdot  p^{k-1-s} + p^{2(k-1-s)} } 	&,	2 \mid \rk(L), \\
            \frac{1 - \chi_{\Df}(p) \left( \frac{-t}{p} \right) p^{-1/2} \cdot p^{k-1-s} }{ 1 - \frac{\Hev_{p}}{p^{k-1}} \cdot p^{k-1-s} + p^{2(k-1-s)}} 	&,	2 \nmid \rk(L) 
          \end{cases}\\
        &\cdot \prod_{2 \neq p \mid \level(L)}
          \frac{1 + K_{L,p} 
          \begin{cases}
            \delta_{\lambda_{p} \neq 0} (1 - p^{-1}) + \delta_{\lambda_{p} = 0} \delta_{p \nmid t}  , & 2 \mid {R_{p}} \\ 
            - \delta_{\lambda_{p} = 0} p^{-1/2} \left(\frac{-t}{p}\right) \epsilon_{p}, & 2 \nmid {R_{p}}
          \end{cases} 
          \cdot p^{k-1-s} + C(\lambda_{p}) \cdot p^{2(k-1-s)}}
          {1 - \left( \frac{\Hev_{p}}{p^{k-1}} - \delta_{2 \mid {R_{p}}} (1-p^{-1}) K_{L,p} \right) p^{k-1-s} + p^{2(k-1-s)}}.
      \end{align*}
      The product converges absolutely for $\Re(s) > k + 1$ as long as the factors
      in the finite product are well defined. The latter is true for
      $\Re(s) > \max\{ \rk(\Df_{p}) \}/2 + k - 1$ and in this range all rational
      factors in $p^{-s}$ are different from zero.
    \end{thm}

	Do note that in the range of absolute convergence specified above, the
    $L$-series vanishes, if, and only if, $\sum_{n \in \N} a(2^n\lambda,2^nt)/2^{ns} = 0$
    in case $2 \mid \level(L)$ or $a(\lambda,t) = 0$ in case $2 \nmid \level(L)$.


	\printbibliography
\end{document}